\documentclass{amsart} 
\usepackage{bm,amsmath,amsthm}     
\usepackage{amssymb,empheq}            
\usepackage{graphicx,enumerate,color}
\usepackage[dvipsnames]{xcolor}
\definecolor{fixedgreen}{rgb}{0.1,0.5,0.1}
\usepackage{tikz} 
\usepackage{tikz-cd}
\usepackage{hyperref}
\usepackage{pagecolor}
\hypersetup{colorlinks,
citecolor=blue,
filecolor=green,
linkcolor=blue,
urlcolor=blue
}

\newlength{\bibitemsep}
\newlength{\bibparskip}
\let\oldthebibliography\thebibliography
\renewcommand\thebibliography[1]{%
  \oldthebibliography{#1}%
  \setlength{\parskip}{\bibitemsep}%
  \setlength{\itemsep}{\bibparskip}%
}

\newtheorem{thm}[subsection]{Theorem}

\newtheorem{prop}[subsection]{Proposition}

\newtheorem{cor}[subsection]{Corollary}
\newtheorem{lemma}[subsection]{Lemma}

\theoremstyle{definition}  
\newtheorem{example}[subsection]{Example}

\newtheorem{warning}[subsection]{Warning}
\newtheorem{remark}[subsection]{Remark}

\newcommand{\dfn}{\textbf} 
\newcommand{\mdfn}[1]{\dfn{\mathversion{bold}#1}} 

\newcommand{\Smash}             {\wedge}

\newcommand{\field}[1]  {\mathbb #1} 
\newcommand{\R}         {\field R}

\newcommand{\Z}         {\field Z}

\newcommand{\F}         {\field F}
\newcommand{\A}         {\field A}

\DeclareMathOperator{\Hom}{Hom}

\DeclareMathOperator{\cok}{cok}

\DeclareMathOperator{\lcm}{lcm}

\newcommand{\Mt}{\mathbb{M}_2}

\newcommand{\Mtb}[1]{\mathbb{M}_2 \langle {#1} \rangle}

\newcommand{\Ct}{C_2}

\newcommand{\pars}[1]{\left({#1}\right)} 

\begin{document}


\title{A Structure Theorem for $RO(C_2)$-graded Bredon Cohomology}

\author{Clover May}

\begin{abstract}
Let $C_2$ be the cyclic group of order two.  We present a structure theorem for the $RO(C_2)$-graded Bredon cohomology of $C_2$-spaces using coefficients in the constant Mackey functor $\underline{\F_2}$.  We show that, as a module over the cohomology of the point, the $RO(C_2)$-graded cohomology of a finite $C_2$-CW complex decomposes as a direct sum of two basic pieces: shifted copies of the cohomology of a point and shifted copies of the cohomologies of spheres with the antipodal action.  The shifts are by elements of $RO(C_2)$ corresponding to actual (i.e.\ non-virtual) $C_2$-representations.  This decomposition lifts to a splitting of genuine $C_2$-spectra.
\end{abstract}

\maketitle

\tableofcontents


\section{Introduction}
For $RO(C_2)$-graded Bredon cohomology, working with coefficients in the constant Mackey functor $\underline{\F_2}$ is the closest analogue to using $\F_2$ coefficients for singular cohomology.  One might expect computations to be fairly straightforward in this setting, as they are in singular cohomology.  Unfortunately, these computations are often nontrivial even for simple $C_2$-spaces.  The goal of this paper is to give a structure theorem for the $RO(C_2)$-graded cohomology of finite $C_2$-CW complexes with coefficients in $\underline{\F_2}$.  This structure theorem can be used to make computations easier.  The corresponding splitting at the spectrum level is a first step toward understanding the category of modules over the genuine equivariant Eilenberg--MacLane spectrum $H{\underline{\F_2}}$.

Let $\Mt$ denote the $RO(C_2)$-graded cohomology of a point.  Let $\A_n$ denote the cohomology of the $n$-dimensional sphere with the antipodal action.  We will show that if $X$ is a finite $C_2$-CW complex, then its cohomology contains only shifted copies of $\Mt$ and shifted copies of $\A_n$ for various $n$.  A bit more precisely, as an $\Mt$-module we can decompose the cohomology of $X$ as
\[
H^{*,*}(X;\underline{\F_2}) \cong (\oplus_i \Sigma^{p_i,q_i} \Mt) \oplus (\oplus_j \Sigma^{r_j,0} \A_{n_j})
\]
for some dimensions $n_j$ and bidegrees $(p_i,q_i)$ and $(r_j,0)$ that correspond to actual (i.e.\ non-virtual) $C_2$-representations so that $0 \leq q_i \leq p_i$ and $0 \leq r_j$.

Each copy of $\Sigma^{p_i,q_i}\Mt$ is the reduced cohomology of a representation sphere $S^{p_i,q_i}$.  So in some sense the structure theorem means the cohomology of any finite $C_2$-CW complex looks like cohomologies of representations spheres and suspensions of antipodal spheres.  At first glance, this might appear obvious because spheres are the building blocks for CW-complexes.  However, we will see it is actually rather surprising that we only need these two types of objects to describe $RO(C_2)$-graded cohomology in $\underline{\F_2}$ coefficients.  The analogous statement for singular cohomology in $\F_2$ coefficients is trivial because the coefficient ring $\F_2$ is a field.  But the coefficient ring $\Mt$ is not a field and there are many $\Mt$-modules that do not appear as the cohomology of a space.  Even the modules that arise in computations as kernels and cokernels of differentials are typically more complicated than simply shifted copies of $\Mt$ and $\A_n$.

As a consequence of the structure theorem we will prove a corresponding splitting of spectra.  In the context of spectra we use the notation $\Sigma^{\infty}X$ and $X$ interchangeably for the genuine equivariant suspension spectrum of a based $C_2$-CW complex $X$. Let $H\underline{\F_2}$ denote the genuine equivariant Eilenberg--MacLane spectrum representing $RO(C_2)$-graded Bredon cohomology with $\underline{\F_2}$-coefficients.  We will show there is a decomposition of $\Sigma^{\infty} X \Smash H\underline{\F_2}$ into a wedge as follows
\[
\Sigma^{\infty} X \Smash H\underline{\F_2} \simeq \pars{\bigvee_i S^{p_i,q_i} \Smash H\underline{\F_2}} \vee \pars{\bigvee_j S^{r_j,0} \Smash {S^{n_j}_a}_+ \Smash H\underline{\F_2}}.
\]
Furthermore if $Y$ is any genuine equivariant finite $C_2$-CW spectrum, we obtain a similar decomposition of $Y \Smash H\underline{\F_2}$, though the shifts may now correspond to virtual representations.

One might hope to generalize the structure theorem in several ways.  Surprisingly, the theorem does not immediately generalize to locally finite $C_2$-CW complexes.  We will see an example of a locally finite but infinite-dimensional $C_2$-CW complex whose cohomology cannot be decomposed as shifted copies of $\Mt$ and shifted copies of $\A_n$.  A complete structure theorem for the cohomology of locally finite $C_2$-CW complexes is not yet known.  Furthermore, the structure theorem presented here only describes the cohomology of a finite $C_2$-CW complex as an $\Mt$-module.  The cohomology as a ring or as a module over the equivariant Steenrod algebra is not well understood.

For computations, it would be useful to have similar structure theorems for the cohomology of $C_2$-CW complexes with coefficients in different Mackey functors.  A structure theorem for $\underline{\F_p}$-coefficients with $p$ an odd prime is trivial and not particularly interesting.  A structure theorem for $\underline{\Z}$-coefficients likely exists but would be considerably more complicated than the one presented here.  Other coefficients such as the Burnside Mackey functor have yet to be explored in this regard.

One might also consider working over other finite groups.  There is currently work in progress toward a structure theorem for $C_p$, the cyclic group of prime order $p$.  In this setting Ferland \cite{F} has computed the cohomology of a point with coefficients in any Mackey functor, building on work of Lewis \cite{L} and unpublished work of Stong.  We expect a somewhat similar structure theorem for the cohomology of finite $C_p$-CW complexes with $\underline{\F_p}$-coefficients will hold.  Unlike $C_2$, however, it is clear such a structure theorem will need to include cohomologies of spaces that are not spheres.  Finding analogous structure theorems for other finite groups remains an open problem.

\subsection{Proof sketch} We briefly outline the proof of the structure theorem from Section \ref{The main theorem} as a guide for the reader.  The ring $\Mt$ is infinite but can be described in terms of particular elements $\rho, \tau$, and $\theta$.  The ring $\A_n$ is isomorphic to $\F_2[\tau,\tau^{-1},\rho]/(\rho^{n+1})$.  Let $X$ be a finite $C_2$-CW complex.  The proof that its cohomology contains only shifted copies of $\Mt$ and shifted copies of $\A_n$
begins by showing that copies of $\Mt$ in $H^{*,*}(X)$ are easily detected by $\theta$.  Accounting for each copy of $\Mt$, we will obtain a short exact sequence of the form
\[
0 \to \oplus_i \Sigma^{p_i,q_i}\Mt \to H^{*,*}(X) \to Q \to 0,
\]
which splits because $\Mt$ is self-injective.  Finally, we will show $Q \cong \oplus_j \Sigma^{r_j,0}\A_{n_j}$.  This will follow from a result about the $\rho$-localization of $H^{*,*}(X)$ and a rather surprising higher decomposition of 1 in $\Mt$ given by the Toda bracket $\langle \tau, \theta, \rho \rangle = 1$.  Together, we will use $\rho$-localization and the Toda bracket to show $Q$ is a finitely generated $\F_2[\tau,\tau^{-1},\rho]$-module, i.e.\ $Q$ is a module over a graded PID.  The graded analogue of the classification of finitely generated modules over a PID completes the proof.

\subsection{Organization of the paper} In Sections \ref{Preliminaries} and \ref{Computational tools} we provide some of the required background and set some notation and terminology.  Section \ref{Graded modules and cohomology} includes several important facts about $\Mt$-modules and their implications for the cohomology of $\Ct$-spaces.
This section includes the computation of the nontrivial Toda bracket of $1$ in $\Mt$ mentioned above.  In Section \ref{The main theorem} we prove the main theorem.  Section \ref{Applications} demonstrates some applications of the main theorem, including the spectrum level splitting.
Appendix \ref{Injectivity} is devoted to the proof of a technical proposition from Section \ref{Graded modules and cohomology}.

\subsection{Acknowledgements}  Much of the work presented here is part of the author's doctoral dissertation at the University of Oregon.  The author would like first and foremost to thank her thesis advisor Dan Dugger for his guidance.  The author would also like to thank Dan Isaksen for suggesting the Toda brackets presented here, Dylan Wilson and Michael Andrews for contributions to the spectrum level splitting, and Eric Hogle for many invaluable conversations.  Finally, thank you to the anonymous referee for helpful comments and suggestions.


\section{Preliminaries}\label{Preliminaries}

We begin with some terminology and notation, much of which can be found in \cite{M}, \cite{S}, \cite{DAH}, and \cite{K}.  Let $G$ be a finite group.  A \mdfn{$G$-CW complex} $X$ is a filtration of a $G$-space formed inductively by attaching orbit cells $G/H \times D^n$ so the filtration quotients are of the form $X_n/X_{n-1} \cong \bigvee_{\alpha} {G/H_{\alpha}}_+ \Smash S^n$.  If the cellular filtration is finite we call $X$ \dfn{finite-dimensional} and refer to the highest dimension as the \mdfn{dimension of $X$}.
We call $X$ \dfn{locally finite} when there are fintely many cells of each dimension.  A \mdfn{finite $G$-CW complex} is both finite-dimensional and locally finite.  We study $G$-CW complexes in part because every CW complex with a cellular $G$-action can be given the structure of a $G$-CW complex (see \cite{M}).

It is often convenient to work with \mdfn{pointed $G$-CW complexes}, also called \mdfn{based $G$-CW complexes}, that have a fixed basepoint.  For an unbased space we can always form $X_+$, i.e.\ $X$ with a disjoint basepoint.  For a based $X$, we write $\Sigma^{\infty}X$ for the genuine equivariant suspension spectrum of $X$.  We will be working with $RO(G)$-graded cohomolgoy of $G$-CW complexes, where the coefficients are given by a Mackey functor.

We now specialize to the case $G = C_2$, the cyclic group of order two, and restrict our focus to $C_2$-CW complexes. We begin by setting some notation for the elements of $RO(C_2)$.  A $p$-dimensional real $C_2$-representation $V$ decomposes as
\[
V \cong (\R^{1,0})^{p-q} \oplus (\R^{1,1})^{q} = \R^{p,q}
\]
where $\R^{1,0}$ is the trivial $1$-dimensional real representation of $C_2$ and $\R^{1,1}$ is the sign representation.  We call $p$ the \dfn{topological dimension} and $q$ the \dfn{weight} or \dfn{twisted dimension} of $V = \R^{p,q}$.  We may also refer to the \dfn{fixed-set dimension}, which is $p-q$.  We use this same notation and terminology when $V$ is a virtual representation given by $V = \R^{p_1,q_1} - \R^{p_2,q_2}$, in which case we say that $V$ has topological dimension $p_1-p_2$, weight $q_1-q_2$, and fixed-set dimension $(p_1 - q_1) - (p_2 - q_2)$.
If $V = \R^{p,q}$ is an actual representation, we write $S^V = S^{p,q}$ for the \dfn{representation sphere} given by the one-point compactification of $V$.  Again, we use this same notation when $V$ is a virtual representation and we are working stably.

For any virtual representation $V = \R^{p,q}$, allowing $p$ and $q$ to be integers, we write $H^V_G(X;M) = H^{p,q}(X;M)$ for the $V$th graded component of the ordinary $RO(C_2)$-graded equivariant cohomology of a $C_2$-space $X$ with coefficients in a Mackey functor $M$.  For a based $C_2$-space,  $\tilde{H}^{*,*}(X;M)$ is the reduced cohomology and $\tilde{H}^{*,*}(X_+;M) = H^{*,*}(X;M)$.  For the rest of the paper we will use coefficients in $\underline{\F_2}$, the constant Mackey functor with value of $\F_2$.
We usually suppress the coefficients and simply write $H^{*,*}(X)$.  When we work non-equivariantly, $H^*_{sing}(X)$ is the singular cohomology with $\F_2$-coefficients of the underlying topological space $U(X)$, where $U$ is the forgetful functor.  The genuine equivariant Eilenberg--MacLane spectrum representing $\tilde{H}^{*,*}(-)$ is $H\underline{\F_2}$.  It has as its underlying spectrum $H\F_2$, the Eilenberg--MacLane spectrum representing $\tilde{H}^{*}_{sing}(-)$.

It is often convenient to plot the bigraded cohomology in the plane.  In these depictions we plot the topological dimension $p$ along the horizontal axis and the weight $q$ along the vertical axis.

\subsection{Cohomology of a point} Using coefficients in the constant Mackey functor $\underline{\F_2}$, the cohomology of a point with the trivial $C_2$-action is the ring denoted by $\Mt := H^{*,*}(pt;\underline{\F_2})$ pictured in Figure \ref{point}.  On the left is a more detailed depiction, though in practice it is easier to work with the more succinct version on the right.  Every lattice point inside the cones represents a copy of the group $\F_2$.  There are unique nonzero elements $\rho \in H^{1,1}(pt)$ and $\tau \in H^{0,1}(pt)$.  As an $\F_2[\rho,\tau]$-module $\Mt$ splits as $\Mt = \Mt^+ \oplus \Mt^-$
where the top cone $\Mt^+$ is a polynomial algebra with generators $\rho$ and $\tau$.  The bottom cone $\Mt^-$ has a unique nonzero element $\theta \in H^{0,-2}(pt)$ that is infinitely divisible by both $\rho$ and $\tau$ and satisfies $\theta^2 = 0$.  We say that every element of the lower cone is $\rho$-torsion, meaning it is zero when multiplied by some power of $\rho$.  Likewise, we say that every element of the lower cone is $\tau$-torsion, since every element is zero when multiplied by some power of $\tau$.  Notice the lower cone is a non-finitely generated ideal, so $\Mt$ is an infinitely generated commutative non-Noetherian $\F_2$-algebra.  This particular depiction first appears in \cite{K} and later in \cite{DGrass}, though the computation precedes these.  Building on unpublished work of Stong, the computation has been reproduced several times and appears in \cite{C} and \cite{F}.

\begin{figure}[ht]
\begin{center}\hfill
\begin{tikzpicture}[scale=0.6]
\draw[gray] (-3.5,0) -- (4.5,0) node[below, black] {\small $p$};
\draw[gray] (0,-4.5) -- (0,4.5) node[left, black] {\small $q$};
\foreach \x in {-3,...,-1,1,2,...,4}
	\draw [font=\tiny, gray] (\x cm,2pt) -- (\x cm,-2pt) node[anchor=north] {$\x$};
\foreach \y in {-4,...,-1,1,2,...,4}
	\draw [font=\tiny, gray] (2pt,\y cm) -- (-2pt,\y cm) node[anchor=east] {$\y$};

\foreach \y in {0,...,4}
	\fill (0,\y) circle(2pt);
\foreach \y in {1,...,4}
	\fill (1,\y) circle(2pt);
\foreach \y in {2,...,4}
	\fill (2,\y) circle(2pt);
\foreach \y in {3,...,4}
	\fill (3,\y) circle(2pt);
\foreach \y in {4,...,4}
	\fill (4,\y) circle(2pt);

\foreach \y in {0,...,2}
	\fill (0,-\y-2) circle(2pt);
\foreach \y in {1,...,2}
	\fill (-1,-\y-2) circle(2pt);
\foreach \y in {2,...,2}
	\fill (-2,-\y-2) circle(2pt);

\draw[thick,->] (0,0) -- (4.5,4.5);
\draw[thick,->] (0,0) -- (0,4.5);
\draw[thick,->] (0,-2) -- (0,-4.5);
\draw[thick,->] (0,-2) -- (-2.5,-4.5);

\draw (0,-0.3) node[below, right] {$1$};
\draw (1,1) node[right] {$\rho$};
\draw (0,1) node[right] {$\tau$};
\draw (0,-2) node[right] {$\theta$};
\draw (-1.1,-3) node[left] {$\frac{\theta}{\rho}$};
\draw (0,-3) node[right] {$\frac{\theta}{\tau}$};

\end{tikzpicture} \hfill
\begin{tikzpicture}[scale=0.6]
\draw[gray] (-3.5,0) -- (4.5,0) node[below, black] {\small $p$};
\draw[gray] (0,-4.5) -- (0,4.5) node[left, black] {\small $q$};
\foreach \x in {-3,...,-1,1,2,...,4}
	\draw [font=\tiny, gray] (\x cm,2pt) -- (\x cm,-2pt) node[anchor=north] {$\x$};
\foreach \y in {-4,...,-1,1,2,...,4}
	\draw [font=\tiny, gray] (2pt,\y cm) -- (-2pt,\y cm) node[anchor=east] {$\y$};
\draw[gray] (2pt, -1cm) -- (-2pt,-1cm);

\draw[thick] (0,0) -- (4.5,4.5);
\draw[thick] (0,0) -- (0,4.5);
\draw[thick] (0,-2) -- (0,-4.5);
\draw[thick] (0,-2) -- (-2.5,-4.5);

\fill (0,0) circle(2pt);
\draw[transparent] (5.5,0) node{$M2$};
\end{tikzpicture}
\end{center}
\caption{$\Mt = H^{*,*}(pt;\underline{\F_2})$.}
\label{point}
\end{figure}

Since there is always an equivariant map $X \to pt$ for any space $X$, its cohomology $H^{*,*}(X)$ is a bigraded $\Mt$-module.  We are interested in cohomology, so throughout this paper we are working in the category of bigraded $\Mt$-modules.  By $\Mt$-module we always mean bigraded $\Mt$-module, and any reference to an $\Mt$-module map means a bigraded homomorphism.  In general, computing the cohomology of a $C_2$-space, even as an $\Mt$-module, is nontrivial.

\subsection{Cohomology of the antipodal sphere}
Let $S^n_a$ denote the $n$-dimensional sphere with the antipodal $C_2$-action.  We write $\A_n$ for the cohomology of $S^n_a$ as an $\Mt$-module.  A picture of $\A_n$ is shown in Figure \ref{An}.  Again, on the left is a more detailed depiction (actually of $\A_4$), while in practice it is more convenient to draw the succinct version on the right.  Here every lattice point in the infinite strip of width $n+1$ represents an $\F_2$.  Diagonal lines represent multiplication by $\rho$ and vertical lines represent multiplication by $\tau$.  Every nonzero element in $\A_n$ is $\rho$-torsion, in the image of $\tau$, and not $\tau$-torsion.  We allow for $n = 0$ since $C_2 = S^0_a$
and the cohomology of $C_2$ can be depicted by a single vertical line.  As a ring $\A_n \cong \F_2[\tau,\tau^{-1},\rho]/(\rho^{n+1})$ where $\rho$ and $\tau$ correspond to multiplication by the usual elements in $\Mt$ and $\tau^{-1}$ has bidegree $(0,-1)$.
The computation of $\A_n$ can be found at the end of Section \ref{Computational tools}.

\begin{figure}[h]
\begin{center}\hfill
\begin{tikzpicture}[scale=0.6]
\draw[gray] (-1,0) -- (5.5,0) node[below, black] {\small $p$};
\draw[gray] (0,-4.3) -- (0,4.3) node[left, black] {\small $q$};
\draw [font=\small, gray] (-0.3,0) node[below] {$0$};
\draw [font=\small, gray] (4.3,0) node[below] {$n$};

\draw (2,5) node {$\vdots$};
\foreach \x in {0,...,4}
	\foreach \y in {-4,...,4}
		\fill (\x,\y) circle(2pt);

\foreach \x in {0,...,4}
	\draw[thick] (\x,-4.3) -- (\x,4.3);
\foreach \y in {-4,...,0}
	\draw[thick] (0,\y) -- (4,\y+4);

\draw[thick] (0,1) -- (3.3,4.3);
\draw[thick] (0,2) -- (2.3,4.3);
\draw[thick] (0,3) -- (1.3,4.3);
\draw[thick] (0,4) -- (0.3,4.3);

\draw[thick] (0.7,-4.3) -- (4,-1);
\draw[thick] (1.7,-4.3) -- (4,-2);
\draw[thick] (2.7,-4.3) -- (4,-3);
\draw[thick] (3.7,-4.3) -- (4,-4);
\draw[thick] (2,-4.6) node {$\vdots$};

\end{tikzpicture}\hfill
\begin{tikzpicture}[scale=0.6]
\draw[gray] (-1,0) -- (5.5,0) node[below, black] {\small $p$};
\draw[gray] (0,-4.3) -- (0,4.3) node[left, black] {\small $q$};
\draw [font=\small, gray] (-0.3,0) node[below] {$0$};
\draw [font=\small, gray] (4.3,0) node[below] {$n$};

\draw[transparent] (2,5) node {$\vdots$};

\foreach \x in {0,4}
	\draw[thick] (\x,-4.3) -- (\x,4.3);
\foreach \y in {-4,...,0}
	\draw[thick] (0,\y) -- (4,\y+4);

\draw[thick] (0,1) -- (3.3,4.3);
\draw[thick] (0,2) -- (2.3,4.3);
\draw[thick] (0,3) -- (1.3,4.3);
\draw[thick] (0,4) -- (0.3,4.3);

\draw[thick] (0.7,-4.3) -- (4,-1);
\draw[thick] (1.7,-4.3) -- (4,-2);
\draw[thick] (2.7,-4.3) -- (4,-3);
\draw[thick] (3.7,-4.3) -- (4,-4);
\draw[transparent] (2,-4.5) node {$\vdots$};

\draw[transparent] (7,0) node {An};
\end{tikzpicture}
\end{center}
\caption{$\A_n = H^{*,*}(S^n_a;\underline{\F_2})$.}
\label{An}
\end{figure}


\section{Computational tools}\label{Computational tools}

In this section we present some common tools for computing $RO(C_2)$-graded cohomology of $C_2$-spaces.  We then apply them to compute $\A_n$, the cohomology of the antipodal sphere.  If $X$ is a $C_2$-CW complex then $X$ has a filtration coming from the cell structure.  The filtration quotients $X_n/X_{n-1}$ are wedges of copies of ${C_2}_+ \Smash S^n$ and $S^{n,0}$ corresponding to the orbit cells that were attached.

More generally, suppose we are given any filtration of a pointed $C_2$-space $X$
\[
pt \subseteq X_0 \subseteq X_1 \subseteq \cdots \subseteq X_k \subseteq X_{k+1} \subseteq \cdots \subseteq X.
\]
Corresponding to the cofiber sequence
\[
X_{k} \hookrightarrow X_{k+1} \to X_{k+1}/X_{k},
\]
for each weight $q$ there is a long exact sequence\footnote{As Kronholm discusses in \cite{K}, these long exact sequences sew together in the usual way to give a spectral sequence for each weight $q$.}
\[
\cdots \to \tilde{H}^{p,q}(X_{k+1}/X_{k}) \to \tilde{H}^{p,q}(X_{k+1}) \to \tilde{H}^{p,q}(X_{k}) \xrightarrow{d} \tilde{H}^{p+1,q}(X_{k+1}/X_{k})\to \cdots.
\]
Taken collectively for all $p$ and $q$, each map in the long exact sequence is a graded $\Mt$-module map.  By abuse, we often refer to the long exact sequences taken collectively for all $q$ as ``the long exact sequence."  Then $d$ is a graded $\Mt$-module map $d: \tilde{H}^{*,*}(X_{k}) \to \tilde{H}^{*+1,*}(X_{k+1}/X_{k})$, which we call ``the differential'' in the long exact sequence.  From this long exact sequence, there is a short exact sequence of graded $\Mt$-modules
\[
0 \to \cok d \to \tilde{H}^{*,*}(X_{k+1}) \to \ker d \to 0.
\]
In many cases the modules $\cok{d}$ and $\ker{d}$ are relatively easily determined, and computing $\tilde{H}^{*,*}(X_{k+1})$ requires solving the extension problem presented in this short exact sequence.

For convenience, we often carry out these computations by plotting the bigraded cohomology in the plane.  Recall that we plot the topological dimension $p$ along the horizontal axis and the weight $q$ along the vertical axis as in the depictions of $\Mt$ and $\A_n$.  With this convention, the differential $d$ in the long exact sequence is depicted by a horizontal arrow since it increases topological dimension by one.  When $\tilde{H}^{*,*}(X_{k})$ is free as an $\Mt$-module the differential $d$ is determined by its value on any set of $\Mt$-generators.

As a special case of the long exact sequence above, we obtain a long exact sequence that relates the ordinary $RO(C_2)$-graded cohomology of a based $C_2$-space $X$ to the non-equivariant singular cohomology of the underlying space $X$.  Consider the cofiber sequence
\[
{C_2}_+ \to S^{0,0} \to S^{1,1}
\]
from \cite{AM}, where the first map is the quotient and the second map is the inclusion of the fixed-set.
Smashing this cofiber sequence with $X$ gives the cofiber sequence
\[
{C_2}_+ \Smash X \to S^{0,0} \Smash X \to S^{1,1} \Smash X
\]
which induces a long exact sequence in cohomology for each weight $q$.  Using this long exact sequence along with the suspension isomorphism and an adjunction\footnote{The suspension isomorphism for $RO(G)$-graded cohomology means that $H^V_G(X;M) \cong \tilde{H}^{V+W}_G(X_+ \Smash S^W;M)$.  The adjunction isomorphism implies $[G_+ \Smash X, Y]_G \cong [X,U(Y)]_e$ for any finite group $G$, where on the left we have homotopy classes of pointed $G$-equivariant maps and on the right homotopy classes of pointed non-equivariant maps between the underlying spaces. See \cite{M}.} we obtain the following lemma as in \cite{K}, originally due to Araki-Murayama \cite{AM}.

\begin{lemma} \label{forgetful les}
(Forgetful long exact sequence).  Let $X$ be a pointed $C_2$-space.  Then for every $q$ there is a long exact sequence
\[
\cdots \to \tilde{H}^{p,q}(X) \xrightarrow{\cdot \rho} \tilde{H}^{p+1,q+1}(X) \xrightarrow{\psi} \tilde{H}^{p+1}_{sing}(X) \to \tilde{H}^{p+1,q}(X) \to \cdots
\]
where $\cdot \rho$ is multiplication by $\rho \in \Mt$ and $\psi: \tilde{H}^{p,q}(X) \to \tilde{H}^{p}_{sing}(X)$ is the forgetful map
to the singular cohomology of the underlying space with $\F_2$ coefficients.
\end{lemma}

The forgetful map preserves topological degree because the underlying spectrum of $H\underline{\F_2}$ is $H\F_2$.  It was shown to be a ring map in \cite{AM} and $\psi: H^{*,*}(pt) \to H^{*}_{sing}(pt)$ sends $\rho \mapsto 0$ and $\tau \mapsto 1$.


\subsection{Computing the cohomology of the antipodal sphere}\label{An proof}

With these tools in place we may now compute the cohomology of $S^n_a$, the $n$-dimensional sphere with the antipodal $C_2$-action.  Using the notation $\A_n = H^{*,*}(S^n_a;\underline{\F_2})$, we will show by induction on $n$ that $\A_n \cong \F_2[\tau,\tau^{-1},\rho]/(\rho^{n+1})$.  For the base case $n=0$ the cohomology of $C_2 = S^0_a$ is part of the computation of $\Mt$ and it is easily seen that
$\A_0 \cong \F_2[\tau,\tau^{-1}]$ (see \cite{C} and \cite{F}).  Now we assume
$\A_{n-1} \cong \F_2[\tau,\tau^{-1},\rho]/(\rho^n)$ and compute $H^{*,*}(S^{n}_a)$ via the cofiber sequence ${S^{n-1}_a}_+ \to {S^n_a}_+ \to {C_2}_+ \Smash S^{n,0}$ where the first map is the inclusion of the equator.  The associated long exact sequence in reduced cohomology is depicted on the left in Figure \ref{An2}.  The only possible differential is zero and we are left with the usual extension problem
\[
0 \to \cok{d} \to H^{*,*}(S^n_a) \to \ker{d} \to 0.
\]
Using the forgetful long exact sequence in Lemma \ref{forgetful les}, we see there is a hidden $\rho$-extension as depicted on the right side of Figure \ref{An2}.  The commutativity of $\rho$ and $\tau$ completes the computation.

\begin{figure}[h]
\begin{center}\hfill
\begin{tikzpicture}[scale=0.6]
\draw[gray] (-1,0) -- (6,0) node[below, black] {\small $p$};
\draw[gray] (0,-4.3) -- (0,4.3) node[left, black] {\small $q$};
\draw [font=\small, gray] (-0.3,0) node[below] {$0$};
\draw [font=\small, gray] (5.3,0) node[below] {$n$};

\draw[->,thick, black] (4,0) -- (5,0);
\draw (4.5,0) node[above] {\small $d$};

\foreach \x in {0,4}
	\draw[thick,red] (\x,-4.3) -- (\x,4.3);
\foreach \y in {-4,...,0}
	\draw[thick,red] (0,\y) -- (4,\y+4);

\draw[thick,red] (0,1) -- (3.3,4.3);
\draw[thick,red] (0,2) -- (2.3,4.3);
\draw[thick,red] (0,3) -- (1.3,4.3);
\draw[thick,red] (0,4) -- (0.3,4.3);

\draw[thick,red] (0.7,-4.3) -- (4,-1);
\draw[thick,red] (1.7,-4.3) -- (4,-2);
\draw[thick,red] (2.7,-4.3) -- (4,-3);
\draw[thick,red] (3.7,-4.3) -- (4,-4);

\draw[thick,blue] (5,-4.3) -- (5,4.3);

\end{tikzpicture}\hfill
\begin{tikzpicture}[scale=0.6]
\draw[gray] (-1,0) -- (6,0) node[below, black] {\small $p$};
\draw[gray] (0,-4.3) -- (0,4.3) node[left, black] {\small $q$};
\draw [font=\small, gray] (-0.3,0) node[below] {$0$};
\draw [font=\small, gray] (5.3,0) node[below] {$n$};

\foreach \x in {0,4}
	\draw[thick] (\x,-4.3) -- (\x,4.3);
\foreach \y in {-4,...,0}
	\draw[thick] (0,\y) -- (4,\y+4);

\draw[thick] (0,1) -- (3.3,4.3);
\draw[thick] (0,2) -- (2.3,4.3);
\draw[thick] (0,3) -- (1.3,4.3);
\draw[thick] (0,4) -- (0.3,4.3);

\draw[thick] (0.7,-4.3) -- (4,-1);
\draw[thick] (1.7,-4.3) -- (4,-2);
\draw[thick] (2.7,-4.3) -- (4,-3);
\draw[thick] (3.7,-4.3) -- (4,-4);

\draw[thick] (5,-4.3) -- (5,4.3);
\draw[thick,dashed] (4,0) -- (5,1);
\draw (4.5,1) node {$\cdot\rho$};

\draw[transparent] (7,0) node {An};
\end{tikzpicture}
\end{center}
\caption{Computing $\A_n = H^{*,*}(S^n_a;\underline{\F_2})$.}
\label{An2}
\end{figure}


\section{Graded modules and cohomology}\label{Graded modules and cohomology}

We now introduce a number of technical lemmas that will aid in the proof of Theorem \ref{main thm}, our main theorem.  First we show that $\theta$ detects copies of $\Mt$ in the sense that an element in any $\Mt$-module with a nonzero $\theta$-multiple generates a free submodule.  Morally this is because there is only one element in $\Mt$ with a nonzero $\theta$-multiple, the generator $1$ of the ring.

\begin{lemma}\label{theta lemma}
Let $N$ be a graded $\Mt$-module containing a nonzero homogeneous element $x$.  If $\theta x$ is nonzero, then $\Mtb{x}$ is a graded free submodule of $N$.
\end{lemma}

\begin{proof}
We will show $\Mtb{x} \subseteq N$ by showing all $\Mt$-multiples of $x$ are nonzero, i.e.\ that $\rho^{m}\tau^{n}x$ and $\frac{\theta}{\rho^m\tau^n}x$ are nonzero for all $m, n \geq 0$.  Since $\theta x$ is nonzero and elements of $\Mt$ commute we have
\[
0 \neq \theta x = \frac{\theta}{\rho^m\tau^n}\cdot \rho^{m}\tau^{n}x = \rho^{m}\tau^{n} \cdot \frac{\theta}{\rho^m\tau^n}x.
\]
This implies $\rho^{m}\tau^{n}x$ and $\frac{\theta}{\rho^m\tau^n}x$ cannot be zero for any nonnegative choice of $m$ or $n$.  So the submodule generated by $x$ is free. \end{proof}

We will also show that $\Mt$ is self-injective, meaning the regular module is injective.  The proof is somewhat tedious and not particularly enlightening.  
The details can be found in Appendix \ref{Injectivity}.

\begin{prop}\label{injectivity prop}
The regular module $\Mt$ is injective as a graded $\Mt$-module.
\end{prop}

So far, the results of this section have been purely algebraic and hold for graded $\Mt$-modules in general.  The next result is specific to the cohomology of a finite $C_2$-CW complex as an $\Mt$-module.  We have already observed a relationship between multiplication by $\rho$ and singular cohomology of the underlying space via the forgetful long exact sequence.  We now show that localization by $\rho$ relates the equivariant cohomology of a space to the singular cohomology of its fixed set.   In the proof of our main theorem, we will see this restricts the types of $\Mt$-modules that can arise as the cohomology of a space.

\begin{lemma}\label{rho localization}
($\rho$-localization) Let $X$ be a finite\footnote{Note that finite-dimensionality is required.  A counterexample that is locally finite but not finite-dimensional is the infinite-dimensional sphere with the antipodal action $S^{\infty}_a$.  It has empty fixed-set but $\A_\infty = H^{*,*}(S^{\infty}_a) \cong \F_2[\tau,\tau^{-1},\rho]$ and $\rho^{-1}\A_\infty$ is nontrivial.} $C_2$-CW complex.  Then
\[
\rho^{-1}H^{*,*}(X) \cong \rho^{-1}H^{*,*}(X^{C_2}) \cong H^{*}_{sing}(X^{C_2}) \otimes_{\F_2} \rho^{-1}\Mt.
\]
\end{lemma}

\begin{proof}
The inclusion $X^{C_2} \xrightarrow{i} X$ induces $\rho^{-1}H^{*,*}(X) \xrightarrow{\rho^{-1}i^*} \rho^{-1}H^{*,*}(X^{C_2})$.  For locally finite, finite-dimensional $C_2$-CW complexes, $\rho^{-1}H^{*,*}(-)$ is a cohomology theory because localization is exact.  On the other hand, $H^{*,*}((-)^{C_2})$ is a cohomology theory because the fixed-set functor $(-)^{C_2}$ preserves Puppe sequences.  So $\rho^{-1}H^{*,*}((-)^{C_2})$
is also a cohomology theory.  It is easily verified that $\rho^{-1}H^{*,*}(-)$ and $\rho^{-1}H^{*,*}((-)^{C_2})$ agree on both orbits, $C_2/C_2 = pt$ and $C_2/e = C_2$, and hence are naturally isomorphic cohomology theories via $\rho^{-1}i^*$.  This proves the first isomorphism above.  The second isomorphism, which relates $\rho$-localization to singular cohomology, follows from the fact that $X^{C_2}$ has trivial action and so has a cellular filtration involving only trivial cells.
\end{proof}

\begin{remark}
An important consequence is that if $X$ is a finite $C_2$-CW complex, then $\rho^{-1}H^{*,*}(X)$ does not have any $\tau$-torsion since $\rho^{-1} \Mt \cong \F_2[\tau,\rho,\rho^{-1}]$ and $\rho^{-1}H^{*,*}(X)$ is free over $\rho^{-1} \Mt$.
\end{remark}

Just as in classical topology, the pairings on $H^{*,*}(-)$ give rise to higher products given by Toda brackets as first defined in \cite{Toda}.  The next proposition involves a higher order decomposition of 1 in the ring $\Mt$.  This result will also restrict the types of $\Mt$-modules that can arise as the cohomology of a space.

\begin{prop}
In $\Mt$, we have the following Toda bracket
\[
\langle \tau, \theta, \rho \rangle = 1
\]
with zero indeterminacy.
\end{prop}

\begin{proof}
First notice that the Toda bracket $\langle \tau, \theta, \rho \rangle$ is well defined since $\theta\rho = 0 = \tau\theta$ in $\Mt$ for degree reasons.  Also notice there is zero indeterminacy because the indeterminacy of the bracket is given by the double coset $\tau H^{0,-1}(pt) + H^{-1,-1}(pt)\rho \equiv 0$.  So $\langle \tau, \theta, \rho \rangle$ is a set containing a single element of $H^{0,0}(pt) \cong \F_2$.  In order to compute $\langle \tau, \theta, \rho \rangle$
we need to determine whether this element is trivial or not.

We will use geometric models for the elements $\tau$, $\theta$, and $\rho$ to prove this Toda bracket is nontrivial.  From \cite{dS}, a model for the $(p,q)$-th Eilenberg--MacLane space representing $\tilde{H}^{p,q}(-; \underline{\F_2})$ is $K(\underline{\F_2}(p,q)) \simeq \F_2 \langle S^{p,q} \rangle$.  This is the usual Dold--Thom model given by configurations of points on $S^{p,q}$ with labels in $\F_2$.  The action on the configurations is inherited from $S^{p,q}$.

We can consider $\rho$ geometrically via
\[
\rho \in H^{1,1}(pt) \cong [S^{0,0},K(\underline{\F_2}(1,1))]_{C_2} \cong [S^{0,0},\F_2\langle S^{1,1} \rangle]_{C_2}.
\]
Or equivalently, using the loop-suspension adjunction
\begin{align*}
\rho \in H^{1,1}(pt) &\cong [S^{0,0},K(\underline{\F_2}(1,1))]_{C_2}\\
&\cong [S^{0,0},\Omega^{1,1}K(\underline{\F_2}(2,2))]_{C_2}\\
&\cong [S^{1,1},K(\underline{\F_2}(2,2))]_{C_2}\\
&\cong [S^{1,1},\F_2\langle S^{2,2} \rangle]_{C_2}.
\end{align*}
On the other hand, we can consider $\theta$ geometrically via
\begin{align*}
\theta \in H^{0,-2}(pt) &\cong [S^{0,0},K(\underline{\F_2}(0,-2))]_{C_2}\\
&\cong [S^{0,0}, \Omega^{2,2}K(\underline{\F_2}(2,0))]_{C_2}\\
&\cong [S^{2,2},K(\underline{\F_2}(2,0))]_{C_2}\\
&\cong [S^{2,2},\F_2 \langle S^{2,0} \rangle]_{C_2}.
\end{align*}
Both $\rho$ and $\theta$ are in the image of the Hurewicz map and factor as
\begin{center}
\begin{tikzcd}
S^{0,0} \arrow[rr, "\rho"] \arrow[swap, dr, "\tilde{\rho}"] & & \F_2 \langle S^{1,1} \rangle \\
 & S^{1,1} \arrow[swap, ur, "\iota"] &
\end{tikzcd}\hfill
\begin{tikzcd}
S^{2,2} \arrow[rr, "\theta"] \arrow[swap, dr, "\tilde{\theta}"] & & \F_2 \langle S^{2,0} \rangle \\
 & S^{2,0} \arrow[swap, ur, "\iota"] &
\end{tikzcd}
\end{center}
where $\iota$ is the canonical map sending each point $x$ to the configuration $[x]$.  Here $\tilde{\rho}$ includes $S^{0,0}$ as the fixed-set of $S^{1,1}$.  This is because $\rho$ is the unique nontrivial element in $H^{1,1}(pt)$ and the composition $\iota \circ \tilde{\rho}$ is not null when restricted to the fixed-sets.  In the following proof we will actually use $\Sigma^{1,1} \tilde{\rho}: S^{1,1} \to S^{2,2}$, the inclusion of a meridian, which factors $\rho$ viewed as a map $\rho : S^{1,1} \to \F_2\langle S^{2,2} \rangle$.
Again $\iota \circ \Sigma^{1,1} \tilde{\rho}$ is not null when restricted to fixed-sets.  Since the target of $\theta$ is fixed by the $C_2$-action, $\theta$ factors through the quotient.  The quotient map $ \tilde{\theta}: S^{2,2} \to S^{2,2}/C_2$ is a degree $2$ map on the underlying sphere $S^2$.

For $\tau$ we observe
\begin{align*}
\tau \in H^{0,1}(pt) &\cong [S^{0,0},K(\underline{\F_2}(0,1))]_{C_2}\\
&\cong [S^{0,0}, \Omega^{1,0}K(\underline{\F_2}(1,1))]_{C_2}\\
&\cong [S^{1,0},K(\underline{\F_2}(1,1))]_{C_2}\\
&\cong [S^{1,0},\F_2 \langle S^{1,1} \rangle]_{C_2}\\
& \cong [S^{2,0},\F_2 \langle S^{2,1} \rangle]_{C_2}.
\end{align*}
We do not actually need a geometric model for $\tau$ to prove the Toda bracket is nontrivial.  We will need the fact that the forgetful map $\psi : H^{*,*}(pt) \to H^*_{sing}(pt)$ sends $\tau \mapsto 1$ in the forgetful long exact sequence from Lemma \ref{forgetful les}.

We are now ready to compute the Toda bracket via the composition
\begin{center}
\begin{tikzcd}
S^{1,1} \arrow[r,"\Sigma^{1,1}\tilde{\rho}"] & S^{2,2} \arrow[r,"\tilde{\theta}"] & S^{2,0} \arrow[r,"\tau"] & \F_2 \langle S^{2,1} \rangle.
\end{tikzcd}
\end{center}
Using the Puppe sequence
\[
S^{1,1} \to S^{2,2} \to {C_2}_+ \Smash S^2 \to S^{2,1} \to \cdots
\]
we can choose maps $f$ and $g$ so the following diagram
\begin{center}
\begin{tikzcd}
S^{1,1} \arrow[r,"\Sigma^{1,1}\tilde{\rho}"] & S^{2,2} \arrow[d] \arrow[r,"\tilde{\theta}"] & S^{2,0} \arrow[r,"\tau"] & \F_2 \langle S^{2,1} \rangle \\
	& {C_2}_+ \Smash S^2 \arrow[d] \arrow[ur, dashed, swap,"f"] & & \\
	& S^{2,1} \arrow[uurr, dashed, swap,"g"] & &
\end{tikzcd}
\end{center}
commutes up to homotopy.  There is no indeterminacy so $\langle \tau, \theta, \rho \rangle = g$.  It remains to show $g$ is not nullhomotopic.  Notice that we can choose $f$ to be the fold map.  Using the adjunction isomorphism, $\tau \circ f$ is an element of the group
\[
[{C_2}_+ \Smash S^2, \F_2 \langle S^{2,1} \rangle]_{C_2} \cong [S^2, \F_2 \langle S^2 \rangle]_e \cong [S^0, \F_2 \langle S^0 \rangle]_e
\]
corresponding to $\psi(\tau)$, which is not null.
The diagram commutes up to homotopy so $g$ cannot be null on the underlying spaces.  Hence $g \simeq \iota$ and $\langle \tau, \theta, \rho \rangle =1$.
\end{proof}

The key to this proof is to find nice geometric models for $\rho$ and $\theta$, and then to recognize we may choose $f$ to be the fold map.  A more algebraic proof suggested by Dan Isaksen makes use of the relationship between Toda brackets and ``hidden extensions.''  More details about this relationship can be found in Section 3.1.1 of \cite{I}.  In particular, we observe the Toda bracket $\langle \tau, \theta, \rho \rangle = 1$ is equivalent to a hidden $\tau$-extension in the cohomology of the cofiber of $\rho$.  Though it is not required, for the sake of consistency with the previous proof, the following discussion demonstrates this hidden extension argument for the cofiber of $\Sigma^{1,1}\tilde{\rho}$.

In the previous argument we observed that $\tau \circ f$ was not null geometrically to deduce that $g$ was not null.  The key to the more algebraic proof is to consider essentially the same diagram and recognize that $\tau \circ f$ is an element of $\tilde{H}^{2,1}({C_2}_+ \Smash S^2)$.  We can observe this element is nonzero by an easy computation in cohomology.  Then we deduce that $g$ is not null as before.

In the Puppe sequence used to compute the Toda bracket above, we have the cofiber sequence $S^{2,2} \to {C_2}_+ \Smash S^2 \to S^{2,1}$.  Associated to this cofiber sequence there is a long exact sequence in cohomology.  By construction the differential will send the generator of $\Sigma^{2,2}\Mt \cong \tilde{H}^{*,*}(S^{2,2})$ to
$\rho$ times the generator of $\Sigma^{2,1}\Mt \cong \tilde{H}^{*,*}(S^{2,1})$ as depicted in Figure \ref{hidden extension}.

\begin{figure}[h]
\begin{center} \hfill\begin{tikzpicture}[scale=0.6]
\draw[gray] (-3.5,0) -- (4.5,0) node[below, black] {\small $p$};
\draw[gray] (0,-4.5) -- (0,4.5) node[left, black] {\small $q$};
\foreach \x in {-3,...,-1,1,2,...,4}
	\draw [font=\tiny, gray] (\x cm,2pt) -- (\x cm,-2pt) node[anchor=north] {$\x$};
\foreach \y in {-4,...,-1,1,2,...,4}
	\draw [font=\tiny, gray] (2pt,\y cm) -- (-2pt,\y cm) node[anchor=east] {$\y$};

\draw[->,thick, black] (2,2) -- (2.8,2);
\draw (2.6,2) node[above] {\small $d$};

\draw[thick, blue] (1.9,1) -- (4.4,3.5);
\draw[thick, blue] (1.9,1) -- (1.9,4.5);
\draw[thick, blue] (1.9,-1) -- (1.9,-4.5);
\draw[thick, blue] (1.9,-1) -- (-1.6,-4.5);
\fill[blue] (1.9,1) circle(2pt);

\draw[thick, red] (2,2) -- (4.5,4.5);
\draw[thick, red] (2,2) -- (2,4.5);
\draw[thick, red] (2,0) -- (2,-4.5);
\draw[thick, red] (2,0) -- (-2.5,-4.5);
\fill[red] (2,2) circle(2pt);


\end{tikzpicture} \hfill
\begin{tikzpicture}[scale=0.6]
\draw[gray] (-3.5,0) -- (4.5,0) node[below, black] {\small $p$};
\draw[gray] (0,-4.5) -- (0,4.5) node[left, black] {\small $q$};

\foreach \x in {-3,...,-1,1,2,...,4}
	\draw [font=\tiny, gray] (\x cm,2pt) -- (\x cm,-2pt) node[anchor=north] {$\x$};
\foreach \y in {-4,...,-1,1,2,...,4}
	\draw [font=\tiny, gray] (2pt,\y cm) -- (-2pt,\y cm) node[anchor=east] {$\y$};

\draw[dashed, thick] (2,0) arc(-65:65:0.53);
\draw (2.6,0.5) node{\small ${\cdot\tau}$};

\draw[thick, red] (2,0) -- (2,-4.5);
\fill[red] (2,0) circle(2pt) node[below left] {\small $\bar{\theta}$};
\draw[red] (3,-2) node {\small $\ker{d}$};

\draw[thick, blue] (2,1) -- (2,4.5);
\fill[blue] (2,1) circle(2pt);
\draw[blue] (3,3) node {\small $\cok{d}$};
\end{tikzpicture}
\end{center}
\caption{Hidden $\tau$ extension in $\tilde{H}^{*,*}({C_2}_+ \Smash S^2)$.}
\label{hidden extension}
\end{figure}

Usually to compute $\tilde{H}^{*,*}({C_2}_+ \Smash S^2)$ we would need to solve the associated extension problem
\[
0 \to \cok{d} \to \tilde{H}^{*,*}({C_2}_+ \Smash S^2) \to \ker{d} \to 0.
\]
But of course we already know from the suspension isomorphism that
\[
\tilde{H}^{*,*}({C_2}_+ \Smash S^2) \cong \Sigma^{2,0} \A_0.
\]
In particular, multiplication by $\tau$ is an isomorphism here.  The element $\theta \in \ker{d}$ contributes a nonzero element $\bar{\theta} \in \tilde{H}^{*,*}({C_2}_+ \Smash S^2)$ and we see the extension problem must be solved by a hidden $\tau$-extension given by $\tau \bar{\theta} \neq 0$.  If we replace $f$ with $\bar{\theta}$, we see this is equivalent to showing $\tau \circ f$ is not null in the previous proof.
Therefore $g$ cannot be null and again we conclude $\langle \tau, \theta, \rho \rangle = 1$.

Armed with this Toda bracket we obtain a matric Toda bracket.\footnote{Matric Toda brackets are defined similarly to matric Massey products, which were first described in \cite{Matric}.}

\begin{lemma}
In $\Mt$, we have the following matric Toda bracket
\[
\left\langle \begin{bmatrix} \rho & \tau \end{bmatrix}, \begin{bmatrix} \tau \\ \rho \end{bmatrix}, \theta \right\rangle = 1
\]
with zero indeterminacy.
\end{lemma}

\begin{proof}
Notice the matric Toda bracket is defined because $\tau \theta = 0 = \rho \theta$.  Again there is zero indeterminacy because $H^{0,2}(pt)\theta \equiv 0$, so the matric Toda bracket is a single element of $H^{0,0}(pt) \cong \F_2$.  Since $\langle \tau, \theta, \rho \rangle = 1$ we can use a juggling formula to shift the bracket and write
\begin{align*}
\left\langle \begin{bmatrix} \rho & \tau \end{bmatrix}, \begin{bmatrix} \tau \\ \rho \end{bmatrix}, \theta \right\rangle \cdot \rho
&= \begin{bmatrix} \rho & \tau \end{bmatrix} \cdot \left\langle \begin{bmatrix} \tau \\ \rho \end{bmatrix}, \theta, \rho \right\rangle \\
&= \rho \cdot \langle \tau, \theta, \rho \rangle + \tau \cdot \langle \rho, \theta, \rho \rangle \\
&= \rho \cdot 1 + \tau \cdot 0 \\
&= \rho
\end{align*}
where $\langle \rho, \theta, \rho \rangle = 0$ for degree reasons.  The matric Toda bracket is an element of $H^{0,0}(pt)$ that is nonzero when multiplied by $\rho$, so it must be nonzero.  This completes the proof.
\end{proof}

Using these two Toda brackets and juggling formulas we get a number of results restricting the types of $\Mt$-modules we can see in cohomology.  We present these results more generally as restrictions on the homotopy of a spectrum.  Recall $H\underline{\F_2}$ denotes the genuine equivariant Eilenberg--MacLane spectrum for $\underline{\F_2}$ so that its bigraded equivariant homotopy is $\pi_{*,*}H\underline{\F_2} = \Mt$.  If $C$ is an $H\underline{\F_2}$-module then $\pi_{*,*}(C)$ is an $\Mt$-module and has Toda brackets.
In each of the following lemmas we take $C$ to be any $H\underline{\F_2}$-module.  In particular, if $X$ is a $C_2$-CW complex, the function spectrum $F(X_+,H\underline{\F_2})$ is an $H\underline{\F_2}$-module, and we can realize $H^{*,*}(X)$ as $\pi_{-*,-*}F(X_+,H\underline{\F_2}) \cong H^{*,*}(X)$.

\begin{lemma}\label{ideal rho tau}
If $x \in \pi_{*,*}(C)$ and $\theta x = 0$, then $x \in (\rho,\tau)\pi_{*,*}(C)$.
\end{lemma}

\begin{proof}
Assume $\theta x = 0$.  Then
\begin{align*}
x = 1 \cdot x
&= \left\langle \begin{bmatrix} \rho & \tau \end{bmatrix}, \begin{bmatrix} \tau \\ \rho \end{bmatrix}, \theta \right\rangle \cdot x \\
&= \begin{bmatrix} \rho & \tau \end{bmatrix} \cdot \left\langle \begin{bmatrix} \tau \\ \rho \end{bmatrix}, \theta, x \right\rangle \\
&= \rho \cdot \langle \tau, \theta, x \rangle + \tau \cdot \langle \rho, \theta, x \rangle,
\end{align*}
which completes the proof.
\end{proof}

\begin{lemma}\label{ideal theta}
If $x \in \pi_{*,*}(C)$ and $\rho x = \tau x = 0$ then $x \in (\theta)\pi_{*,*}(C)$.
\end{lemma}

\begin{proof}
Assume $\rho x = \tau x = 0$.  Then
\[
x = x \cdot 1 = x \cdot \left\langle \begin{bmatrix} \rho & \tau \end{bmatrix}, \begin{bmatrix} \tau \\ \rho \end{bmatrix}, \theta \right\rangle = \left\langle x, \begin{bmatrix} \rho & \tau \end{bmatrix}, \begin{bmatrix} \tau \\ \rho \end{bmatrix} \right\rangle \cdot \theta,
\]
which completes the proof.
\end{proof}

\begin{lemma}\label{ideal rho}
If $x \in \pi_{*,*}(C)$ and $\tau x = 0$ then $x \in (\rho)\pi_{*,*}(C)$.
\end{lemma}

\begin{proof}
Assume $\tau x = 0$.  Then $x = 1 \cdot x = x \cdot \langle \tau, \theta, \rho \rangle = \langle x, \tau, \theta \rangle \cdot \rho$ so we are done.
\end{proof}

\begin{lemma}\label{ideal tau}
If $x \in \pi_{*,*}(C)$ and $\rho x = 0$ then $x \in (\tau)\pi_{*,*}(C)$.
\end{lemma}

\begin{proof}
The proof is analogous to the proof of Lemma \ref{ideal rho}.
\end{proof}

Next we observe two vanishing regions in the cohomology of any finite $C_2$-CW complex.  These regions are depicted on the left side of Figure \ref{vanishing fig}.

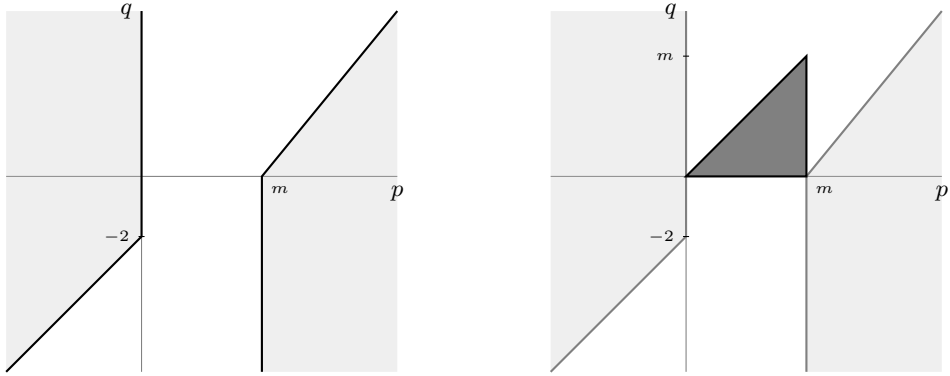
\begin{figure}[ht]
\begin{tikzpicture}[scale=0.4]

\draw[gray] (-4.5,0) -- (8.5,0);
\draw[gray] (0,-6.5) -- (0,5.5);

\fill[lightgray, opacity=0.25] (0,5.5) -- (0,-2) -- (-4.5,-6.5) -- (-4.5,5.5) -- cycle;
\draw[thick] (0,5.5) -- (0,-2) -- (-4.5,-6.5);

\fill[lightgray, opacity=0.25] (4,-6.5) -- (4,0) -- (8.5,5.5) -- (8.5,-6.5) -- cycle;
\draw[thick] (4,-6.5) -- (4,0) -- (8.5,5.5);

\draw (0.1,-2) -- (-0.1,-2) node[left] {\tiny $-2$};
\draw (8.5,0) node[below, black] {\small $p$};
\draw (0,5.5) node[left, black] {\small $q$};
\draw (4,0) node[below right] {\tiny $m$};

\end{tikzpicture}\hfill
\begin{tikzpicture}[scale=0.4]

\draw[gray] (-4.5,0) -- (8.5,0);
\draw[gray] (0,-6.5) -- (0,5.5);

\fill[lightgray, opacity=0.25] (0,5.5) -- (0,-2) -- (-4.5,-6.5) -- (-4.5,5.5) -- cycle;
\draw[thick, gray] (0,5.5) -- (0,-2) -- (-4.5,-6.5);

\fill[lightgray, opacity=0.25] (4,-6.5) -- (4,0) -- (8.5,5.5) -- (8.5,-6.5) -- cycle;
\draw[thick, gray] (4,-6.5) -- (4,0) -- (8.5,5.5);

\fill[gray] (0,0) -- (4,0) -- (4,4) -- cycle;
\draw[thick] (0,0) -- (4,0) -- (4,4) -- cycle;

\draw (0.1,-2) -- (-0.1,-2) node[left] {\tiny $-2$};
\draw (8.5,0) node[below, black] {\small $p$};
\draw (0,5.5) node[left, black] {\small $q$};
\draw (0.1,4) -- (-0.1,4) node[left] {\tiny $m$};
\draw (4,0) node[below right] {\tiny $m$};

\end{tikzpicture}
\caption{Vanishing regions and region containing $\Mt$ generators.}\label{vanishing fig}
\end{figure}

\begin{lemma}\label{vanishing}
If $X$ is a finite $C_2$-CW complex of dimension $m$ then $H^{p,q}(X) = 0$
\begin{enumerate}
\item whenever $p < 0$ and $q > p-2$, and
\item whenever $p > m$ and $q < p-m$.
\end{enumerate}
\end{lemma}

\begin{proof}
Both statements follow easily by induction on the $C_2$-CW filtration for $X$ since the cohomologies of the orbits $\Mt = H^{*,*}(pt)$ and $\A_0 = H^{*,*}(C_2)$ satisfy these vanishing regions.
\end{proof}

An immediate corollary restricts the bidegree of a generator for a shifted copy of $\Mt$ in $H^{*,*}(X)$.  The region where $\Mt$ generators can lie is depicted by the triangle on the right side of Figure \ref{vanishing fig}.

\begin{cor}\label{bidegree corollary}
Let $X$ be a finite $C_2$-CW complex with dimension $m$.  Any generator for a copy of $\Mt$ in $H^{*,*}(X)$ must lie in a bidegree $(p,q)$ satisfying $0 \leq p \leq m$ and $0 \leq q \leq p$.  Thus $(p,q)$ corresponds to an actual representation.
\end{cor}

\begin{proof}
The proof follows immediately from Lemma \ref{vanishing} since otherwise a copy of $\Mt$ would intersect one of the vanishing regions.
\end{proof}

The last lemma in this section is key to the proof of the main theorem as we now show that if $\theta$ acts trivially on a nice submodule of $H^{*,*}(X)$, then every element is not $\tau$-torsion and is infinitely divisible by $\tau$, making the submodule an $\F_2[\tau,\tau^{-1},\rho]$-module.  The hypotheses of the following lemma are somewhat technical.  They guarantee this submodule is the homotopy of an $H\underline{\F_2}$-module so we may apply the various Toda bracket results.  We also require the submodule to be a direct summand of the cohomology of a finite $C_2$-CW complex in order to use both $\rho$-localization and the vanishing regions.

\begin{lemma}\label{graded PID}
Let $X$ be a finite $C_2$-CW complex and let $C$ be an $H\underline{\F_2}$-module with $\pi_{*,*}(C) \subseteq H^{*,*}(X)$. Suppose that, as an $\Mt$-module, $\pi_{*,*}(C)$ is a direct summand of $H^{*,*}(X)$.  Furthermore, suppose this summand is killed by $\theta$ so that $\theta x = 0$ for all $x \in \pi_{*,*}(C)$.
Then $\cdot \tau: \pi_{*,*}(C) \to \pi_{*,*}(C)$ is an automorphism, making $\pi_{*,*}(C)$ naturally an $\F_2[\tau,\tau^{-1},\rho]$-module.
\end{lemma}

\begin{proof}
First we show that multiplication by $\tau$ is injective.  If there exists a nonzero element $x \in \pi_{*,*}(C)$ with $\tau x = 0$, then there are two cases, either $x$ is $\rho$-torsion or $x$ survives $\rho$-localization.  Both cases lead to a contradiction.

\begin{enumerate}
\item Suppose $x$ is $\rho$-torsion so that $\rho^n x = 0$ for some $n$.  If $n=1$ then $\rho x =0$ and $\tau x = 0$.  Then by Lemma \ref{ideal theta}, $x$ is in the image of multiplication by $\theta$, contradicting that $\theta$ acts trivially on $\pi_{*,*}(C)$.
If $\rho^{n}x = 0$ but $\rho^{n-1}x \neq 0$ then $\rho^{n-1}x$ is killed by $\tau$ since $\rho$ and $\tau$ commute:
\[
\tau \rho^{n-1}x = \rho^{n-1} \tau x = 0.
\]
Again by Lemma \ref{ideal theta}, $\rho^{n-1}x$ is in the image of multiplication by $\theta$, a contradiction.
\item Suppose $x$ is not $\rho$-torsion so that $x$ survives $\rho$-localization.  Recall from Lemma \ref{rho localization}, any element surviving the $\rho$-localization of the cohomology of a finite space cannot be $\tau$-torsion.  Since $\pi_{*,*}(C)$ is a summand of $H^{*,*}(X)$, $x$ cannot be $\tau$-torsion, which contradicts the assumption that $\tau x = 0$.
\end{enumerate}
So indeed, $\tau x$ is nonzero for all $x \in \pi_{*,*}(C)$ and the map $\cdot \tau: \pi_{*,*}(C) \to \pi_{*,*}(C)$ is injective.  

Notice that injectivity of $\cdot \tau$ means that any nonzero element $x \in \pi_{*,*}(C)$ is not $\tau$-torsion so that $\tau^m x \neq 0$ for all $m$.  In particular, $x$ cannot be in a bidegree with negative topological dimension.  Otherwise, some $\tau$-multiple of $x$ would land in the first vanishing region and contradict Lemma \ref{vanishing}.  We will use this fact in the proof of surjectivity.

To show multiplication by $\tau$ is surjective, we assume to the contrary there is some nonzero homogenous element $y \in \pi_{*,*}(C)$ not in the image of $\cdot \tau$.  We may further assume that $y$ is an element of minimal topological dimension satisfying these hypotheses.  We can make this minimality assumption because, as we have already observed, injectivity of $\cdot \tau$ implies all elements of $\pi_{*,*}(C)$ have nonnegative topological dimension.

Since $\theta y = 0$, Lemma \ref{ideal rho tau} implies that $y \in (\rho, \tau)\pi_{*,*}(C)$.  So we can write $y = \rho a + \tau b$ for some homogeneous elements $a, b \in \pi_{*,*}(C)$.  Notice that $a \neq 0$, otherwise $y=\tau b$, contradicting our assumption that $y$ is not in the image of $\cdot \tau$.  The topological dimension of $a$ is one less than the topological dimension of $y$.
Since $y$ had minimal topological dimension and $a$ is nonzero, it must be that $a$ is in the image of $\cdot \tau$.  This means we can write $a = \tau c$ for some $c \in \pi_{*,*}(C)$, but now $y = \rho a + \tau b = \rho \tau c + \tau b = \tau (\rho c + \tau b)$, contradicting that $y$ is not in the image of $\cdot \tau$.

Surjectivity of $\cdot \tau$ completes the proof that multiplication by $\tau$ is an automorphism of $\pi_{*,*}(C)$ and so $\pi_{*,*}(C)$ is a $\F_2[\tau, \tau^{-1}, \rho]$-module.
\end{proof}


\section{The main theorem}\label{The main theorem}

We are now ready to state and prove the main theorem.

\begin{thm}\label{main thm}
For any finite\footnote{Note that finiteness is again required here.  One might hope for a generalization to locally finite $C_2$-CW complexes by allowing shifted copies of $\A_\infty$.  A counterexample is given at the end of this section in Example \ref{counterexample}.} $C_2$-CW complex $X$, there is a decomposition of the $RO(C_2)$-graded cohomology of $X$ with constant $\underline{\F_2}$-coefficients as
\[
H^{*,*}(X;\underline{\F_2}) \cong (\oplus_i \Sigma^{p_i,q_i} \Mt) \oplus (\oplus_j \Sigma^{r_j,0} \A_{n_j})
\]
as a module over $\Mt = H^{*,*}(pt;\underline{\F_2})$.  Here each $\R^{p_i,q_i}$ and $\R^{r_j,0}$ are elements of $RO(C_2)$ corresponding to actual representations so that $0 \leq q_i \leq p_i$ and $0 \leq r_j$.
\end{thm}

\begin{proof}
Recall that $\theta$ detects copies of $\Mt$.  So if $x \in H^{*,*}(X)$ with $\theta x \neq 0$, then by Lemma \ref{theta lemma} there is a free submodule $\Mtb{x} \subseteq H^{*,*}(X)$.  The short exact sequence of $\Mt$-modules
\[
0 \to \Mtb{x} \to H^{*,*}(X) \to P \to 0
\]
splits because $\Mt$ is injective (see Proposition \ref{injectivity prop}).  Continuing to split off summands in this way, we have the split short exact sequence
\[
0 \to \oplus_i \Sigma^{p_i,q_i} \Mt \to H^{*,*}(X) \to Q \to 0
\]
where we can assume that every $x \in Q$ satisfies $\theta x = 0$.  This process terminates because by induction any given bidegree of $H^{*,*}(X)$ is finite-dimensional as a vector space.  Moreover, by Corollary \ref{bidegree corollary}, each bidegree $(p_i,q_i)$ corresponds to an actual representation so that $0 \leq q_i \leq p_i$.

We would now like to apply Lemma \ref{graded PID} to $Q$ to show that $Q$ is an $\F_2[\tau, \tau^{-1}, \rho]$-module.  However, in order to apply the lemma we need to realize $Q$ as the homotopy of an $H\underline{\F_2}$-module.  We can realize $H^{*,*}(X)$ as the homotopy of a function spectrum $H^{*,*}(X) \cong \pi_{-*,-*}F(X_+,H\underline{\F_2})$.
Since $F(X_+,H\underline{\F_2})$ is an $H\underline{\F_2}$-module, each free generator $u_i$ of $\Sigma^{p_i,q_i} \Mt$ in $H^{*,*}(X)$ gives rise to a map of spectra given by the composition
\[
S^{-p_i,-q_i} \Smash H\underline{\F_2} \xrightarrow{u_i \Smash id} F(X_+,H\underline{\F_2}) \Smash H\underline{\F_2} \xrightarrow{\mu} F(X_+,H\underline{\F_2}).
\]
Taking maps corresponding to each of the free generators of the shifted copies of $\Mt$ in $H^{*,*}(X)$, we obtain a map
\[
\bigvee_i \pars{S^{-p_i,-q_i} \Smash H\underline{\F_2}} \to F(X_+,H\underline{\F_2}).
\]
Let $C$ be the cofiber of this map.  The cofiber sequence induces a long exact sequence in homotopy of the form
\[
\cdots \to \pi_{*+1,*} C \to \oplus_i \Sigma^{p_i,q_i} \Mt \to H^{*,*}(X) \to \pi_{*,*}C \to \cdots
\]
where by construction $\Sigma^{p_i,q_i} \Mt \hookrightarrow H^{*,*}(X)$ is the inclusion of the copies of $\Mt$ from our original short exact sequence.  Because this is an inclusion and $\Mt$ is self-injective, we get a split short exact sequence and $\pi_{*,*}C \cong Q$.  We have now realized $Q$ as the homotopy of an $H\underline{\F_2}$-module, so Lemma \ref{graded PID} applies and indeed $Q$ is an $\F_2[\tau,\tau^{-1},\rho]$-module.

We still need to show $Q$ is finitely generated as an $\F_2[\tau,\tau^{-1},\rho]$-module so we may decompose $Q$ according to the graded version of the structure theorem for finitely generated modules over a graded PID.  Observe that $\tau^{-1}\Mt \cong \F_2[\tau,\tau^{-1},\rho]$ and that $\tau^{-1}H^{*,*}(X)$ is finitely generated as a $\tau^{-1} \Mt$-module.
The latter follows from induction on the $C_2$-CW filtration for $X$ since $\tau^{-1}\Mt$ is a graded PID and the $\tau$-localized cohomology of each orbit is finitely generated as a $\tau^{-1}\Mt$-module.  The submodule $Q \cong \tau^{-1}Q$ of $\tau^{-1}H^{*,*}(X)$ is also finitely generated as a $\tau^{-1}\Mt$-module since $\tau^{-1}\Mt$ is Noetherian.

Finally, applying the fundamental theorem for finitely generated graded modules over a graded PID to $Q$, it must be the case that
\[
Q \cong \tau^{-1}Q \cong (\oplus_k \Sigma^{p_k,q_k} \tau^{-1}\Mt) \oplus (\oplus_j \Sigma^{r_j,0} \tau^{-1}\Mt / (\rho^{n_j + 1}))
\]
where again we identify $\tau^{-1} \Mt \cong \F_2[\tau,\tau^{-1},\rho]$.  However, $X$ is a finite $C_2$-CW complex, so the second vanishing region from Lemma \ref{vanishing} implies $Q$ cannot contain any summands of the form $\Sigma^{p_k,q_k} \tau^{-1}\Mt$.  This means
\[
Q \cong \tau^{-1}Q \cong \oplus_j \Sigma^{r_j,0} \tau^{-1}\Mt / (\rho^{n_j + 1}) \cong \oplus_j \Sigma^{r_j,0} \A_{n_j} 
\]
Finally we can conclude that
\[
H^{*,*}(X;\underline{\F_2}) \cong (\oplus_i \Sigma^{p_i,q_i} \Mt) \oplus Q \cong (\oplus_i \Sigma^{p_i,q_i} \Mt) \oplus (\oplus_j \Sigma^{r_j,0} \A_{n_j})
\]
as desired.
\end{proof}

The reader may notice we used the condition that $X$ was finite several times in the previous proof.  At first glance, one might expect a generalization of the structure theorem to locally finite $C_2$-CW complexes if we allow shifted copies of $\A_\infty$.  This seems plausible because $S^\infty_a$ is locally finite.  As further evidence, $\A_\infty \cong \tau^{-1}\Mt \cong \F_2[\tau,\tau^{-1},\rho]$ appears near the end of the proof as part of the fundamental theorem for finitely generated graded modules over a graded PID.  The following counterexample demonstrates that such a generalization would also need to include other types of $\Mt$-modules.

\begin{example}\label{counterexample}
In this example we consider an infinite-dimensional locally finite $C_2$-CW complex whose cohomology is not $\A_\infty$.  Consider the $C_2$-CW complex $S^{\infty,\infty}$ formed by the colimit of the diagram
\[
S^{0,0} \to S^{1,1} \to S^{2,2} \to \cdots \to S^{n,n} \to \cdots
\]
where each map is a suspension of $\tilde{\rho}:S^{0,0} \to S^{1,1}$.  Notice that $S^{\infty,\infty}$ can be given a cell structure with two fixed points and a single equivariant $n$-cell $C_2 \times D^n$ for every $n>0$.  Alternatively, $S^{\infty,\infty}$ can be realized as the unreduced suspension of $S^{\infty}_a$.  Its cohomology is depicted in Figure \ref{cohomology counterexample}.  More precisely, we can describe $\tilde{H}^{*,*}(S^{\infty,\infty})$ as $\Sigma^{0,-1} N$, where $N$ is the quotient in the category of graded $\Mt$-modules in the short exact sequence
\[
0 \to \Mt[\rho^{-1}] \to \Mt[\tau^{-1}\rho^{-1}] \to N \to 0.
\]
\end{example}

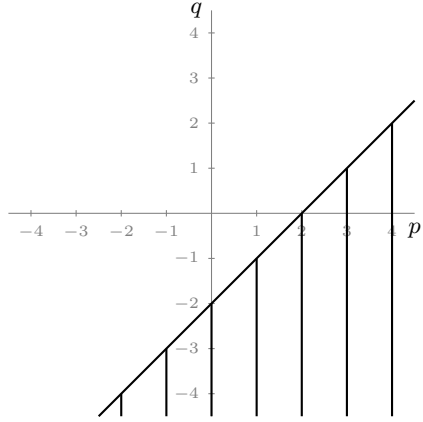
\begin{figure}[h]
\begin{center}
\begin{tikzpicture}[scale=0.6]
\draw[gray] (-4.5,0) -- (4.5,0) node[below, black] {\small $p$};
\draw[gray] (0,-4.5) -- (0,4.5) node[left, black] {\small $q$};
\foreach \x in {-4,...,-1,1,2,...,4}
	\draw [font=\tiny, gray] (\x cm,2pt) -- (\x cm,-2pt) node[anchor=north] {$\x$};
\foreach \y in {-4,...,-1,1,2,...,4}
	\draw [font=\tiny, gray] (2pt,\y cm) -- (-2pt,\y cm) node[anchor=east] {$\y$};

\draw[thick] (-2.5,-4.5) -- (4.5,2.5);
\foreach \x in {-2,...,4}
  \draw[thick] (\x,\x-2) -- (\x,-4.5);

\draw[transparent] (5.5,0) node {RP2};
\end{tikzpicture}
\end{center}
\caption{Cohomology of $S^{\infty,\infty}$.}
\label{cohomology counterexample}
\end{figure}


\section{Applications}\label{Applications}

In this section we present some applications of the main theorem, including a topological version given by a spectrum level splitting.  We first present some applications to computations in $RO(C_2)$-graded cohomology.

\subsection{Computational Applications}
We begin with some examples that illustrate common computational techniques and demonstrate how the main theorem simplifies computations.  In the first example we will use the following fact, which implies we can compute the $p$-axis of the $RO(C_2)$-graded cohomology\footnote{The $p$-axis of the $RO(C_2)$-graded cohomology is the $\Z$-graded equivariant cohomology originally defined by Bredon in \cite{Br}.} of a space using the singular cohomology of the quotient.
\begin{lemma} \label{quotient lemma}
Let $X$ be a $C_2$-space.  Then $H^{p,0}(X) \cong H^p_{sing}(X/C_2)$.
\end{lemma}

\begin{proof}
This follows from working with coefficients in a constant Mackey functor.  Recall from \cite{dS}, a model for the $(p,q)$-th Eilenberg--MacLane space is given by $K(\underline{\F_2}(p,q)) \simeq \F_2 \langle S^{p,q} \rangle$, the usual Dold--Thom model given by configurations of points on $S^{p,q}$ with labels in $\F_2$.  In particular, $K(\underline{\F_2}(p,0)) \simeq \F_2 \langle S^{p,0} \rangle$.
Since $S^{p,0}$ is fixed, so is the space $\F_2 \langle S^{p,0} \rangle$.  By adjunction
\begin{align*}
H^{p,0}(X) &\cong [X_+, K(\underline{\F_2}(p,0))]_{C_2}\\
&\cong [X_+, \F_2 \langle S^{p,0} \rangle]_{C_2}\\
&\cong [X_+/{C_2}, U(\F_2 \langle S^{p,0} \rangle)]_e\\
&\cong [X_+/{C_2}, \F_2 \langle S^{p} \rangle)]_e\\
&\cong H^p_{sing}(X/C_2),
\end{align*}
where the last isomorphism follows from the non-equivariant Dold--Thom model.
\end{proof}

\begin{example}\label{real proj example}
In this example we compute the cohomology of the projective space $\R P^2_{tw} = \mathbb{P}(\R^{3,1})$ using Lemma \ref{quotient lemma} and our main theorem.  A picture of $\R P^2_{tw}$ is shown in Figure \ref{real proj fig}.  This is the usual diagram for $\R P^2$ given by a disk with opposite points on the boundary identified.  The $C_2$-action is given by rotating the picture $180^\circ$ leaving a fixed point in the center and a fixed circle on the boundary.

\begin{figure}[h]
\begin{center}
\begin{tikzpicture}[scale=0.8]
\draw[thick, fixedgreen,-] (0,0) arc (-90:90:1cm);
\draw[thick,->] (1,1.1) -- (1,1.12);
\draw[thick, fixedgreen,-] (0,2) arc (90:270:1cm);
\draw[thick,->] (-1,0.98) -- (-1,0.96);
\draw[->] (-0.175,1) arc (-180:90:5pt);
\fill[fixedgreen] (0,0.01) circle(1.75pt);
\fill[fixedgreen] (0,2) circle(1.75pt);
\fill[fixedgreen] (0,1) circle(1.75pt);

\draw[transparent] (1,0) node {stuff};
\end{tikzpicture}
\end{center}
\caption{$\R P^2_{tw}$.}
\label{real proj fig}
\end{figure}

The long exact sequence associated to the cofiber sequence $S^{1,0} \hookrightarrow \R P^2_{tw} \to S^{2,2}$ is depicted on the left side of Figure \ref{differential real proj}.  Recall that in these depictions every lattice point inside the cones represents a copy of $\F_2$ and the differential increases topological dimension by one.

\begin{figure}[h]
\begin{center}\hfill
\begin{tikzpicture}[scale=0.6]
\draw[gray] (-3.5,0) -- (4.5,0) node[below, black] {\small $p$};
\draw[gray] (0,-4.5) -- (0,4.5) node[left, black] {\small $q$};
\foreach \x in {-3,...,-1,1,2,...,4}
	\draw [font=\tiny, gray] (\x cm,2pt) -- (\x cm,-2pt) node[anchor=north] {$\x$};
\foreach \y in {-4,...,-1,1,2,...,4}
	\draw [font=\tiny, gray] (2pt,\y cm) -- (-2pt,\y cm) node[anchor=east] {$\y$};

\draw[->,thick, black] (1,0) -- (1.95,0);
\draw (1.6,0) node[above] {\small $d$};

\draw[thick, red] (1,0) -- (4.5,3.5);
\draw[thick, red] (1,0) -- (1,4.5);
\draw[thick, red] (1,-2) -- (1,-4.5);
\draw[thick, red] (1,-2) -- (-1.5,-4.5);
\fill[red] (1,0) circle(2pt);

\draw[thick, blue] (2,2) -- (4.5,4.5);
\draw[thick, blue] (2,2) -- (2,4.5);
\draw[thick, blue] (2,0) -- (2,-4.5);
\draw[thick, blue] (2,0) -- (-2.5,-4.5);
\fill[blue] (2,2) circle(2pt);


\end{tikzpicture}\hfill
\begin{tikzpicture}[scale=0.6]
\draw[gray] (-3.5,0) -- (4.5,0) node[below, black] {\small $p$};
\draw[gray] (0,-4.5) -- (0,4.5) node[left, black] {\small $q$};
\foreach \x in {-3,...,-1,1,2,...,4}
	\draw [font=\tiny, gray] (\x cm,2pt) -- (\x cm,-2pt) node[anchor=north] {$\x$};
\foreach \y in {-4,...,-1,1,2,...,4}
	\draw [font=\tiny, gray] (2pt,\y cm) -- (-2pt,\y cm) node[anchor=east] {$\y$};

\draw[thick, red] (1,1) -- (2,2) -- (2,1) -- (4.5,3.5);
\draw[thick, red] (1,1) -- (1,4.5);
\draw[thick, red] (1,-2) -- (1,-4.5);
\draw[thick, red] (1,-2) -- (-1.5,-4.5);
\draw[red] (3.5,1.5) node {\small $\ker{d}$};

\draw[thick, blue] (2,2.1) -- (4.5,4.6);
\draw[thick, blue] (2,2.1) -- (2,4.5);
\draw[thick, blue] (2,-0.9) -- (2,-4.5);
\draw[thick, blue] (2,-0.9) -- (1,-1.9) -- (1,-0.9) -- (-2.6,-4.5);
\draw[blue] (3,-2) node {\small $\cok{d}$};

\draw[transparent] (5.5,0) node {RP2};
\end{tikzpicture}
\end{center}
\caption{Differential in a long exact sequence for $\tilde{H}^{*,*}(\R P^2_{tw}$).}\label{differential real proj}
\end{figure}
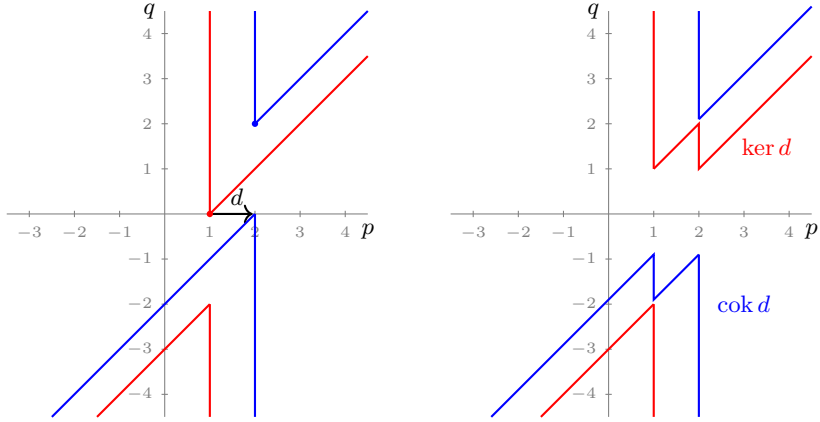
%
%
%
%

The only possible differential, determined by its image on the generator of the free module $\tilde{H}^{*,*}(S^{1,0}) \cong  \Sigma^{1,0}\Mt$, must be nonzero by Lemma \ref{quotient lemma}.  This is because the quotient $\R P^2_{tw}/C_2$ is the cone on $S^1$, which is contractible.
The $\Mt$-modules $\cok{d}$ and $\ker{d}$ resulting from this differential are depicted on the right side of Figure \ref{differential real proj}.   Even though we know this differential, computing $\tilde{H}^{*,*}(\R P^2_{tw})$ requires solving the extension problem in the short exact sequence
\[
0 \to \cok{d} \to \tilde{H}^{*,*}(\R P^2_{tw}) \to \ker{d} \to 0.
\]
One might hope for the short exact sequence to be split, but this turns out not to be the case.

From Theorem \ref{main thm}, the cohomology of $\R P^2_{tw}$ contains only shifted copies of $\Mt$ and $\A_n$.  Looking at the modules $\cok{d}$ and $\ker{d}$ we see a gap.  Along the $p$-axis the cohomology is trivial because the quotient $\R P^2_{tw}/C_2$ is contractible.  Since $\A_n$ has no gaps, the cohomology of $\R P^2_{tw}$ must consist only of copies of $\Mt$.
Finally by inspecting the rank in each bidegree, we see that $\tilde{H}^{*,*}(\R P^2_{tw}) \cong \Sigma^{1,1} \Mt \oplus \Sigma^{2,1} \Mt$.
\end{example}

In general, many $\Mt$-modules arise as kernels and cokernels of differentials in spectral sequence or long exact sequence computations.  One advantage of our main theorem is to resolve the associated extension problems.  In Example \ref{real proj example} we saw a differential between two copies of $\Mt$ leading to the $\Mt$-modules $\cok{d}$ and $\ker{d}$ depicted in Figure \ref{differential real proj} and discovered the result must consist only of copies of $\Mt$.
Here the generator of one copy of $\Mt$ hit an element of the lower cone of the other $\Mt$ and the result was two copies of $\Mt$ shifted vertically from their original position.  
Indeed, any differential from the generator of a copy of $\Mt$ to the lower cone of another copy gives rise to a similar result, as seen in \cite{K}.

We present a few more examples of differentials to give an idea of the range of possible $\Mt$-modules one might encounter in computations.  These examples also illustrate how our main theorem can be used.

\begin{example}
In Figure \ref{differential upper cone} we see a differential where the generator of one copy of $\Mt$ hits an element in the upper cone of another, rather than the lower cone as in Example \ref{real proj example}.  This differential $d$ is depicted on the left side of Figure \ref{differential upper cone} with $\cok{d}$ and $\ker{d}$ pictured on the right.
Every element here is $\tau$-torsion but both $\Mt$ and $\A_n$ have elements that are not $\tau$-torsion.  This means the usual extension problem cannot be solved with any number of copies of $\Mt$ or $\A_n$.  As a consequence of Theorem \ref{main thm}, a differential like this one is impossible when computing the cohomology of a finite $C_2$-CW complex.

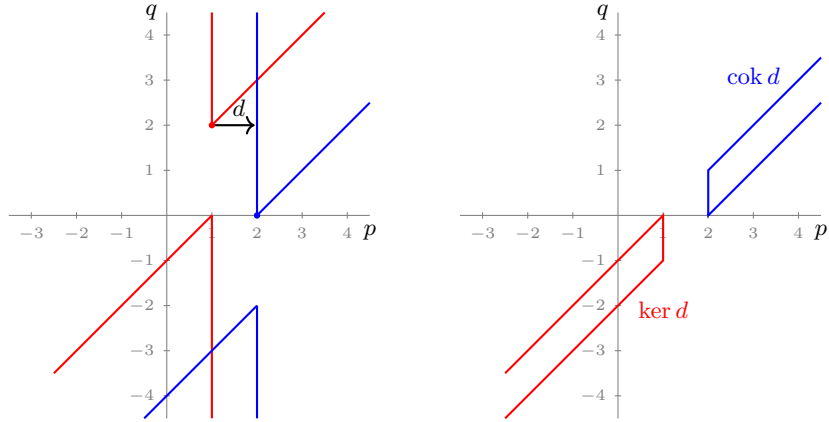
\begin{figure}[h]
\begin{center} \hfill
\begin{tikzpicture}[scale=0.6]
\draw[gray] (-3.5,0) -- (4.5,0) node[below, black] {\small $p$};
\draw[gray] (0,-4.5) -- (0,4.5) node[left, black] {\small $q$};

\foreach \x in {-3,...,-1,1,2,...,4}
	\draw [font=\tiny, gray] (\x cm,2pt) -- (\x cm,-2pt) node[anchor=north] {$\x$};
\foreach \y in {-4,...,-1,1,2,...,4}
	\draw [font=\tiny, gray] (2pt,\y cm) -- (-2pt,\y cm) node[anchor=east] {$\y$};

\draw[->,thick] (1,2) -- (1.95,2);
\draw (1.6,2) node[above] {\small $d$};

\draw[thick, red] (1,2) -- (3.5,4.5);
\draw[thick, red] (1,2) -- (1,4.5);
\draw[thick, red] (1,0) -- (1,-4.5);
\draw[thick, red] (1,0) -- (-2.5,-3.5);
\fill[red] (1,2) circle(2pt);

\draw[thick, blue] (2,0) -- (4.5,2.5);
\draw[thick, blue] (2,0) -- (2,4.5);
\draw[thick, blue] (2,-2) -- (2,-4.5);
\draw[thick, blue] (2,-2) -- (-0.5,-4.5);
\fill[blue] (2,0) circle(2pt);

\end{tikzpicture} \hfill
\begin{tikzpicture}[scale=0.6]
\draw[gray] (-3.5,0) -- (4.5,0) node[below,black] {\small $p$};
\draw[gray] (0,-4.5) -- (0,4.5) node[left,black] {\small $q$};

\foreach \x in {-3,...,-1,1,2,...,4}
	\draw [font=\tiny, gray] (\x cm,2pt) -- (\x cm,-2pt) node[anchor=north] {$\x$};
\foreach \y in {-4,...,-1,1,2,...,4}
	\draw [font=\tiny, gray] (2pt,\y cm) -- (-2pt,\y cm) node[anchor=east] {$\y$};

\draw[thick, red] (-2.5,-3.5) -- (1,0) -- (1,-1) -- (-2.5,-4.5);
\draw[red] (1,-2.1) node {\small $\ker{d}$};

\draw[thick, blue] (4.5,3.5) -- (2,1) -- (2,0) -- (4.5,2.5);
\draw[blue] (3,3.1) node {\small $\cok{d}$};

\draw[transparent] (5.5,0) node {diff};
\end{tikzpicture}
\end{center}
\caption{Differential to upper cone of $\Mt$.}
\label{differential upper cone}
\end{figure}

\end{example}

\begin{warning}
It is in fact possible to have a nonzero differential into the upper cone when computing the cohomology of a space.  For example, one could have an isomorphism between two copies of $\Mt$.  One could also have the generator of one copy hit $\rho^n$ times the generator of the other copy.  Neither of these contradict our main theorem.  The next example demonstrates the latter type of differential.
\end{warning}

\begin{example}
Consider the space $X = S^{2,2} \cup I_{\text{triv}}$ where a line segment with the trivial action connects the north and south poles of $S^{2,2}$.  There is a cofiber sequence $S^{2,2} \hookrightarrow X \to S^{1,0}$ and the differential in the long exact sequence associated to this cofiber sequence is depicted in Figure \ref{second differential upper cone}.
This differential must be nonzero by Lemma \ref{rho localization} because the fixed-set of $X$ is contractible so $\rho^{-1}\tilde{H}^{*,*}(X) \equiv 0$.  Here the generator of one copy of $\Mt$ hits $\rho^2$ times the generator of the other copy. As a result of the main theorem, the associated extension problem in this example must be solved by $\Sigma^{1,0}\A_{1}$.

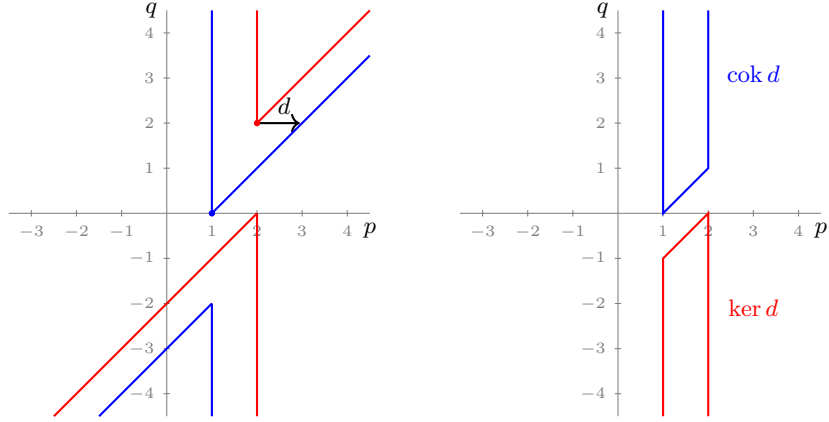
\begin{figure}[h]
\begin{center} \hfill
\begin{tikzpicture}[scale=0.6]
\draw[gray] (-3.5,0) -- (4.5,0) node[below, black] {\small $p$};
\draw[gray] (0,-4.5) -- (0,4.5) node[left, black] {\small $q$};

\foreach \x in {-3,...,-1,1,2,...,4}
	\draw [font=\tiny, gray] (\x cm,2pt) -- (\x cm,-2pt) node[anchor=north] {$\x$};
\foreach \y in {-4,...,-1,1,2,...,4}
	\draw [font=\tiny, gray] (2pt,\y cm) -- (-2pt,\y cm) node[anchor=east] {$\y$};

\draw[->,thick] (2,2) -- (2.95,2);
\draw (2.6,2) node[above] {\small $d$};

\draw[thick, red] (2,2) -- (4.5,4.5);
\draw[thick, red] (2,2) -- (2,4.5);
\draw[thick, red] (2,0) -- (2,-4.5);
\draw[thick, red] (2,0) -- (-2.5,-4.5);
\fill[red] (2,2) circle(2pt);

\draw[thick, blue] (1,0) -- (4.5,3.5);
\draw[thick, blue] (1,0) -- (1,4.5);
\draw[thick, blue] (1,-2) -- (1,-4.5);
\draw[thick, blue] (1,-2) -- (-1.5,-4.5);
\fill[blue] (1,0) circle(2pt);

\end{tikzpicture} \hfill
\begin{tikzpicture}[scale=0.6]
\draw[gray] (-3.5,0) -- (4.5,0) node[below, black] {\small $p$};
\draw[gray] (0,-4.5) -- (0,4.5) node[left, black] {\small $q$};

\foreach \x in {-3,...,-1,1,2,...,4}
	\draw [font=\tiny, gray] (\x cm,2pt) -- (\x cm,-2pt) node[anchor=north] {$\x$};
\foreach \y in {-4,...,-1,1,2,...,4}
	\draw [font=\tiny, gray] (2pt,\y cm) -- (-2pt,\y cm) node[anchor=east] {$\y$};

\draw[thick, red] (1,-4.5) -- (1,-1) -- (2,0) -- (2,-4.5);
\draw[red] (3,-2.1) node {\small $\ker{d}$};

\draw[thick, blue] (1,4.5) -- (1,0) -- (2,1) -- (2,4.5);
\draw[blue] (3,3.1) node {\small $\cok{d}$};

\draw[transparent] (5.5,0) node {diff};
\end{tikzpicture}
\end{center}
\caption{Differential in a long exact sequence for $\tilde{H}^{*,*}(S^{2,2} \cup I_\text{triv})$.}
\label{second differential upper cone}
\end{figure}

\end{example}

\begin{example}
Another example of a differential $d$ that arises in computations is from the generator of a copy of $\Mt$ to a shifted copy of $\A_0$.
As a trivial example, one could compute $\tilde{H}^{*,*}(S^{2,1})$ via the cofiber sequence $S^{1,0} \hookrightarrow S^{2,1} \to {C_2}_+ \Smash S^2$ as depicted in Figure \ref{differential line}.
The differential here is nontrivial because the quotient $S^{2,1}/C_2$ is contractible.  Counting the rank in each bidegree, Theorem \ref{main thm} implies the extension problem in the associated short exact sequence
\[
0 \to \cok{d} \to \tilde{H}^{*,*}(S^{2,1}) \to \ker{d} \to 0
\]
must be solved by $\tilde{H}^{*,*}(S^{2,1}) \cong \Sigma^{2,1}\Mt$, which agrees with the suspension isomorphism.

\begin{figure}[h]
\begin{center}\hfill
\begin{tikzpicture}[scale=0.6]
\draw[gray] (-3.5,0) -- (4.5,0) node[below, black] {\small $p$};
\draw[gray] (0,-4.5) -- (0,4.5) node[left, black] {\small $q$};

\foreach \x in {-3,...,-1,1,2,...,4}
	\draw [font=\tiny, gray] (\x cm,2pt) -- (\x cm,-2pt) node[anchor=north] {$\x$};
\foreach \y in {-4,...,-1,1,2,...,4}
	\draw [font=\tiny, gray] (2pt,\y cm) -- (-2pt,\y cm) node[anchor=east] {$\y$};

\draw[->,thick] (1,0) -- (1.95,0);
\draw (1.6,0) node[above] {\small $d$};

\draw[thick, red] (1,0) -- (4.5,3.5);
\draw[thick, red] (1,0) -- (1,4.5);
\draw[thick, red] (1,-2) -- (1,-4.5);
\draw[thick, red] (1,-2) -- (-1.5,-4.5);
\fill[red] (1,0) circle(2pt);

\draw[thick, blue] (2,-4.5) -- (2,4.5);

\end{tikzpicture} \hfill
\begin{tikzpicture}[scale=0.6]
\draw[gray] (-3.5,0) -- (4.5,0) node[below, black] {\small $p$};
\draw[gray] (0,-4.5) -- (0,4.5) node[left, black] {\small $q$};

\foreach \x in {-3,...,-1,1,2,...,4}
	\draw [font=\tiny, gray] (\x cm,2pt) -- (\x cm,-2pt) node[anchor=north] {$\x$};
\foreach \y in {-4,...,-1,1,2,...,4}
	\draw [font=\tiny, gray] (2pt,\y cm) -- (-2pt,\y cm) node[anchor=east] {$\y$};

\draw[thick, red] (2,1) -- (4.5,3.5);
\draw[thick, red] (2,1) -- (2,4.5);
\draw[thick, red] (1,-2) -- (1,-4.5);
\draw[thick, red] (1,-2) -- (-1.5,-4.5);
\draw[thick, red] (3.5,1.5) node {\small $\ker{d}$};

\draw[thick, blue] (2,-1) -- (2,-4.5);
\draw[thick, blue] (3,-2) node {\small $\cok{d}$};

\draw[transparent] (5.5,0) node {diff};
\end{tikzpicture}
\end{center}
\caption{Differential in a long exact sequence for $\tilde{H}^{*,*}(S^{2,1})$.}
\label{differential line}
\end{figure}
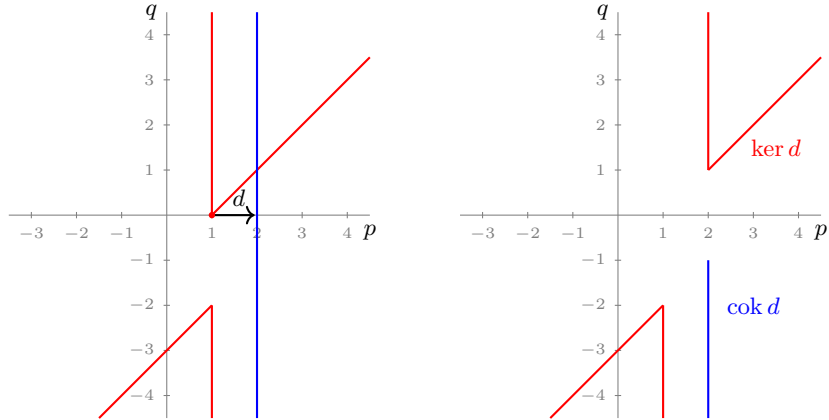
\end{example}

\subsection{Structure theorem for homology}  We obtain a similar structure theorem for $RO(C_2)$-graded homology as an immediate consequence of our main theorem.  Since $\Mt$ is self-injective, we can define an $RO(C_2)$-graded homology theory via graded $\Mt$-module maps
\[
H_{a,b}(X) = \Hom_{\Mt}(H^{*,*}(X),\Sigma^{a,b} \Mt).
\]
One can check this homology theory agrees with the usual $RO(C_2)$-graded Bredon homology with $\underline{\F_2}$-coefficients on each orbit.
Let $\Mt^* = \Hom_{\Mt}(\Mt,\Sigma^{*,*} \Mt)$ denote the homology of a point and $\A_{n}^* = \Hom_{\Mt}(\A_{n}, \Sigma^{*,*} \Mt)$ denote the homology of $S^n_a$.  These $\Mt$-modules are depicted in Figure \ref{duals}, with $\Mt^*$ on the left and $\A_n^*$ on the right.
These are very similar to the depictions of $\Mt$ and $\A_n$, but now $\rho$ acts with bidegree $(-1,-1)$, $\tau$ acts with bidegree $(0,-1)$, $\theta$ acts with bidegree $(0,2)$, and so on.  The same depiction of $\Mt^*$ can be found in Figure 1 of \cite{DI}.

\begin{figure}[h]
\begin{center}\hfill
\begin{tikzpicture}[scale=0.6]
\draw[gray] (-4.5,0) -- (3.5,0) node[below, black] {\small $p$};
\draw[gray] (0,-4.5) -- (0,4.5) node[left, black] {\small $q$};
\foreach \x in {-4,...,-1,1,2,...,3}
	\draw [font=\tiny, gray] (\x cm,2pt) -- (\x cm,-2pt) node[anchor=north] {$\x$};
\foreach \y in {1,2,...,4}
	\draw [font=\tiny, gray] (2pt,\y cm) -- (-2pt,\y cm) node[anchor=east] {$\y$};
  \foreach \y in {-4,...,-1}
  	\draw [font=\tiny, gray] (2pt,\y cm) -- (-2pt,\y cm) node[anchor=west] {$\y$};

\foreach \y in {0,...,4}
	\fill (0,-\y) circle(2pt);
\foreach \y in {1,...,4}
	\fill (-1,-\y) circle(2pt);
\foreach \y in {2,...,4}
	\fill (-2,-\y) circle(2pt);
\foreach \y in {3,...,4}
	\fill (-3,-\y) circle(2pt);
\foreach \y in {4,...,4}
	\fill (-4,-\y) circle(2pt);

\foreach \y in {0,...,2}
	\fill (0,\y+2) circle(2pt);
\foreach \y in {1,...,2}
	\fill (1,\y+2) circle(2pt);
\foreach \y in {2,...,2}
	\fill (2,\y+2) circle(2pt);

\draw[thick,->] (0,0) -- (-4.5,-4.5);
\draw[thick,->] (0,0) -- (0,-4.5);
\draw[thick,->] (0,2) -- (0,4.5);
\draw[thick,->] (0,2) -- (2.5,4.5);


\draw[transparent] (2,5) node {$\vdots$};
\draw[transparent] (2,-4.6) node {$\vdots$};
\end{tikzpicture} \hfill
\begin{tikzpicture}[scale=0.6]
\draw[gray] (-1,0) -- (5.5,0) node[below, black] {\small $p$};
\draw[gray] (0,-4.3) -- (0,4.3) node[left, black] {\small $q$};
\draw [font=\small, gray] (-0.3,0) node[below] {$0$};
\draw [font=\small, gray] (4.3,0) node[below] {$n$};

\draw (2,5) node {$\vdots$};
\foreach \x in {0,...,4}
	\foreach \y in {-4,...,4}
		\fill (\x,\y) circle(2pt);

\foreach \x in {0,...,4}
	\draw[thick] (\x,-4.3) -- (\x,4.3);
\foreach \y in {-4,...,0}
	\draw[thick] (0,\y) -- (4,\y+4);

\draw[thick] (0,1) -- (3.3,4.3);
\draw[thick] (0,2) -- (2.3,4.3);
\draw[thick] (0,3) -- (1.3,4.3);
\draw[thick] (0,4) -- (0.3,4.3);

\draw[thick] (0.7,-4.3) -- (4,-1);
\draw[thick] (1.7,-4.3) -- (4,-2);
\draw[thick] (2.7,-4.3) -- (4,-3);
\draw[thick] (3.7,-4.3) -- (4,-4);
\draw[thick] (2,-4.6) node {$\vdots$};

\draw[transparent] (6,0) node {diff};
\end{tikzpicture}
\end{center}
\caption{$\Mt^{*} = H_{*,*}(pt;\underline{\F_2})$ and $\A_n^{*} = H_{*,*}(S^n_a;\underline{\F_2})$.}
\label{duals}
\end{figure}

Now from the main theorem we immediately obtain a decomposition of the homology of any finite $C_2$-CW complex as an $\Mt$-module.

\begin{cor}\label{homology} As an $\Mt$-module, there is a decomposition of the homology of any finite $C_2$-CW complex $X$ given by
\[
H_{*,*}(X) \cong (\oplus_i \Sigma^{p_i,q_i} \Mt^*) \oplus (\oplus_j \Sigma^{r_j,0} \A_{n_j}^*)
\]
where $(p_i,q_i)$ and $(r_j,0)$ correspond to elements of $RO(C_2)$ with $0 \leq q_i \leq p_i$ and $0 \leq r_j$.
\end{cor}

\subsection{Spectrum level splitting} We further extend this result to a topological version of the structure theorem.   We show we can lift the decomposition of $RO(C_2)$-graded homology from Corollary \ref{homology} to a splitting at the spectrum level.  In particular, we will show for any based finite $C_2$-CW complex $X$ that we have a decomposition of the genuine $C_2$-spectrum $\Sigma^{\infty}X \Smash H\underline{\F_2}$ given by
\[
\Sigma^{\infty} X \Smash H\underline{\F_2} \simeq \pars{\bigvee_i S^{p_i,q_i} \Smash H\underline{\F_2}} \vee \pars{\bigvee_j S^{r_j,0} \Smash {S^{n_j}_a}_+ \Smash H\underline{\F_2}}.
\]
Here each wedge summand corresponds to the appropriate direct summand in $\tilde{H}_{*,*}(X)$.
Recall that we often use $X$ and $\Sigma^{\infty}X$ interchangeably, so for example, in this setting ${S^{n_j}_a}_+$ means $\Sigma^{\infty}{S^{n_j}_a}_+$.  

Lifting the direct sum decomposition to a spectrum level statement relies primarily on the fact that for genuine $C_2$-spectra, bigraded homotopy groups detect weak equivalences.\footnote{One can show this using the cofiber sequence ${C_2}_+ \to S^{0,0} \to S^{1,1}$ and extending to the Puppe sequence, then applying the five lemma.}  In order to prove the spectrum level statement we will begin with a smaller case.

\begin{lemma}\label{An homotopy}
Let $X$ be a $C_2$-CW spectrum with bigraded homology given by $\pi_{*,*}(X \Smash H\underline{\F_2}) \cong \A_n^*$.  Then $X \Smash H\underline{\F_2} \simeq {S^n_a}_+ \Smash H\underline{\F_2}$.
\end{lemma}

\begin{proof}
We aim to construct a map ${S^n_a}_+ \Smash H\underline{\F_2} \to X \Smash H\underline{\F_2}$ that induces an isomorphism on bigraded homotopy.  As above, such a map will be sufficient because we are working with genuine $C_2$-spectra and in that setting $RO(C_2)$-graded homotopy detects weak equivalences.

Since $\pi_{*,*}(X \Smash H\underline{\F_2}) \cong \A_n^*$, there is a class $x \in \pi_{n,n+1}(X \Smash H\underline{\F_2})$ with $\rho^{n+1} x = 0$ but $\rho^n x \neq 0$.
Consider $x$ as a map
\[
x: S^{n+1,n+1} \to S^{1,0} \Smash X \Smash H\underline{\F_2}.
\]
Since $\rho^{n+1}x=0$, the map $x$ factors through the cofiber of $\rho^{n+1}$ as in the diagram below.
\begin{center}
\begin{tikzcd}
S^{0,0} \arrow[r,"\rho^{n+1}"] & S^{n+1,n+1} \arrow[r,"x"] \arrow[d] & S^{1,0} \Smash X \Smash H\underline{\F_2}\\
& S^{1,0} \Smash {S^n_a}_+ \arrow[ur,dashed,"f"]
\end{tikzcd}
\end{center}
Now the induced map $f$ factors as the unit map followed by an $H\underline{\F_2}$-module map $\bar{f}$ as in the following diagram.
\begin{center}
\begin{tikzcd}
S^{0,0} \arrow[r,"\rho^{n+1}"] & S^{n+1,n+1} \arrow[r,"x"] \arrow[d] & S^{1,0} \Smash X \Smash H\underline{\F_2}\\
& S^{1,0} \Smash {S^n_a}_+ \arrow[ur,dashed,"f"] \arrow[r,"\eta"] & S^{1,0} \Smash {S^n_a}_+ \Smash H\underline{\F_2} \arrow[u,"\bar{f}"]
\end{tikzcd}
\end{center}
Since $\bar{f}$ is a map of $H\underline{\F_2}$-modules inducing an isomorphism on homotopy in bidegree $(n+1,n+1)$, it must induce an isomorphism in all bidegrees.  Thus $\bar{f}$ is a weak equivalence, which completes the proof.
\end{proof}

\begin{remark}
The weight of the class $x$ in the homotopy of $X \Smash H\underline{\F_2}$ was not actually important for the proof.  A similar argument shows that
\[
S^{p,0} \Smash {S^n_a}_+ \Smash H\underline{\F_2} \simeq S^{p,q} \Smash {S^n_a}_+ \Smash H\underline{\F_2}
\]
for any $q$.  Notice, however, it is not the case in general that
\[
S^{p,0} \Smash {S^n_a}_+ \simeq S^{p,q} \Smash {S^n_a}_+.
\]
For example, $S^{1,0} \Smash {S^1_a}_+$ is not equivalent to $S^{1,1} \Smash {S^1_a}_+$.  The lack of equivalence here can be detected by smashing with $H\underline{\Z}$.
\end{remark}

We will use Lemma \ref{An homotopy} to prove the following theorem.

\begin{thm}\label{spectra}
Let $X$ be a based finite $C_2$-CW complex.  Then there is a weak equivalence of genuine $C_2$-spectra
\[
\Sigma^{\infty} X \Smash H\underline{\F_2} \simeq \pars{\bigvee_i S^{p_i,q_i} \Smash H\underline{\F_2}} \vee \pars{\bigvee_j S^{r_j,0} \Smash {S^{n_j}_a}_+ \Smash H\underline{\F_2}}
\]
where $(p_i,q_i)$ and $(r_j,0)$ correspond to elements of $RO(C_2)$ with $0 \leq q_i \leq p_i$ and $0 \leq r_j$.
\end{thm}

\begin{proof}
As in Lemma \ref{An homotopy} the goal is to construct a map
\[
\pars{\bigvee_i S^{p_i,q_i} \Smash H\underline{\F_2}} \vee \pars{\bigvee_j S^{r_j,0} \Smash {S^{n_j}_a}_+ \Smash H\underline{\F_2}} \to \Sigma^{\infty} X \Smash H\underline{\F_2}
\]
that induces an isomorphism on bigraded homotopy.  Again such a map is sufficient because for genuine $C_2$-spectra the $RO(C_2)$-graded homotopy detects weak equivalences.

By Corollary \ref{homology}, we may decompose $\pi_{*,*}(\Sigma^{\infty} X \Smash H\underline{\F_2}) \cong \tilde{H}_{*,*}(X; \underline{\F_2})$ as
\[
\tilde{H}_{*,*}(X; \underline{\F_2}) \cong (\oplus_i \Sigma^{p_i,q_i}\Mt^*) \oplus (\oplus_j \Sigma^{r_j,0} \A_{n_j}^*).
\]
For each $i$, take a class $u_i \in \pi_{p_i,q_i}(\Sigma^{\infty} X \Smash H\underline{\F_2})$ corresponding to the generator of the $i$th shifted copy of $\Mt^*$ in this decomposition.  Then for each $j$, take a class $x_j \in \pi_{n_j + r_j,n_j+1}(\Sigma^{\infty} X \Smash H\underline{\F_2})$, such that $\rho^{n_j+1} x_j = 0$ but $\rho^{n_j} x_j \neq 0$.
The choice of $x_j$ corresponds to the choice of $x$ in Lemma \ref{An homotopy} for each shifted copy of $\A_{n_j}^*$.  Then, just as in Lemma \ref{An homotopy}, each $x_j$ gives rise to a map $\bar{f_j}$ as shown in the diagram below.
\begin{center}
\begin{tikzcd}
S^{0,0} \arrow[r,"\rho^{n_j+1}"] & S^{n_j+1,n_j+1} \arrow[r,"x"] \arrow[d] & S^{1-r_j,0} \Smash \Sigma^{\infty} X \Smash H\underline{\F_2}\\
& S^{1,0} \Smash {S^{n_j}_a}_+ \arrow[ur,dashed,"f_j"] \arrow[r,"\eta"] & S^{1,0} \Smash {S^{n_j}_a}_+ \Smash H\underline{\F_2} \arrow[u,"\bar{f_j}"]
\end{tikzcd}
\end{center}
Now take the maps the $u_i$ and $\bar{f_j}$ all together to form the map
\[
\pars{\bigvee_i S^{p_i,q_i} \Smash H\underline{\F_2}} \vee \pars{\bigvee_j S^{r_j,0} \Smash {S^{n_j}_a}_+ \Smash H\underline{\F_2}} \xrightarrow{\pars{\bigvee_i u_i\,} \vee\, \pars{\bigvee_j \bar{f_j}}} \Sigma^{\infty} X \Smash H\underline{\F_2}.
\]
By construction, this map induces an isomorphism on bigraded homotopy and so is a weak equivalence.
\end{proof}

\begin{remark}
We have two immediate consequences of Theorem \ref{spectra}.  The first is that the structure theorems for homology and cohomology hold Mackey functor valued.\footnote{As presented here $H^{p,q}(X;\underline{\F_2})$ is an abelian group but one can also define Mackey functor valued $RO(C_2)$-graded cohomology with $\underline{\F_2}$-coefficients.
The value at $C_2/C_2$ is given by $H^{p,q}(X;\underline{\F_2})$ and the value at $C_2/e$ is given by $H^{p,q}(C_2 \times X;\underline{\F_2})$.  The maps in the diagram for the Mackey functor are induced by the projection $C_2 \times X \to X$.  Since the structure theorems lift to the spectrum level, they hold for Mackey functor valued homology and cohomology, as modules over the Mackey functor valued cohomology of a point $\underline{\Mt}$.} Furthermore, if $Y$ is a finite $C_2$-CW spectrum, it is a desuspension of the suspension spectrum of a based finite $C_2$-CW complex.
Thus we obtain a similar decomposition of $Y \Smash H\underline{\F_2}$ as a wedge sum
\[
Y \Smash H\underline{\F_2} \simeq \pars{\bigvee_i S^{p_i,q_i} \Smash H\underline{\F_2}} \vee \pars{\bigvee_j S^{r_j,0} \Smash {S^{n_j}_a}_+ \Smash H\underline{\F_2}}
\]
where now the shifts may correspond to virtual representations.
\end{remark}

\subsection{Borel cohomology} Another application of the main theorem can be observed via Borel equivariant cohomology.  For a $C_2$-space $X$, by $H^*_{Bor}(X)$ we mean the Borel equivariant cohomology of $X$ with $\F_2$-coefficients.  This can be computed as $H^*_{sing}(X \times_{C_2} EC_2) \cong H^*_{sing}(X \times_{C_2} S^{\infty}_a)$ using $S^{\infty}_a$ as a model for $EC_2$.

Using the fiber bundle map $X \times_{C_2} EC_2 \to BC_2$, $H^*_{Bor}(X)$ is a module over the graded $\F_2$-algebra $H^*_{sing}(BC_2) = H^*_{sing}(\R P^{\infty}) \cong \F_2[x]$, where $x$ has degree $1$.  If $X$ is a finite $C_2$-CW complex, $H^*_{Bor}(X)$ is a finitely generated module over a graded PID and there is a decomposition
\[
H^*_{Bor}(X) \cong \left(\oplus_k \Sigma^{a_k} \F_2[x] \right) \oplus \left( \oplus_\ell \Sigma^{b_\ell} \F_2[x]/(x^{n_\ell}) \right)
\]
as a graded $\F_2[x]$-module.

The following result relating the $\tau$-localization of the $RO(C_2)$-graded cohomology with Borel cohomology is well known.
\begin{lemma}\label{Borel}
For any finite $C_2$-CW complex $X$, identifying $x$ with $\tau^{-1}\rho$, we have the following isomorphism of $\F_2[x]$-modules
\[
(\tau^{-1}H)^{*,0}(X) \cong H^*_{Bor}(X).
\]
\end{lemma}

\begin{proof}
The proof involves showing $H^{*,*}((-) \times EC_2)$ and $\tau^{-1}H^{*,*}(-)$ are isomorphic cohomology theories for finite $C_2$-CW complexes, and then restricting the isomorphism.
Let $X$ be a finite $C_2$-CW complex and consider the equivariant projection
$\pi: X \times EC_2 \to X$.  Since $X \times EC_2$ has a free $C_2$ action it can be built with cells of the form $C_2 \times D^n$ making $H^{*,*}(X \times EC_2)$ an $\F_2[\tau,\tau^{-1}]$-module.  Thus the map $\pi^*$ on cohomology factors through $\tau^{-1}H^{*,*}(X)$ as in the diagram below.
\begin{center}
\begin{tikzcd}
H^{*,*}(X) \arrow[rr, "\pi^*"] \arrow[dr] & & H^{*,*}(X \times EC_2) \\
& \tau^{-1}H^{*,*}(X) \arrow[swap, dashed, ur, "\varphi"] &
\end{tikzcd}
\end{center}
For finite $C_2$-CW complexes $\tau^{-1}H^{*,*}(-)$ is a cohomology theory since localization is exact.  Now since $(-) \times EC_2$ preserves Puppe sequences, $H^{*,*}((-) \times EC_2)$ is also a cohomology theory.  It is easy to verify the map $\varphi$ is an isomorphism for both orbits since $pt \times EC_2=S^{\infty}_a$ and $C_2 \times EC_2 \simeq C_2$.
This shows $H^{*,*}((-) \times EC_2)$ and $\tau^{-1}H^{*,*}(-)$ are isomorphic cohomology theories for finite $C_2$-CW complexes.

We now restrict to the $p$-axis of each theory.  Identifying $x$ with $\tau^{-1}\rho$, the map $\varphi$ restricts to an isomorphism of $\F_2[x]$-modules as in the following diagram.
\begin{center}
\begin{tikzcd}
H^{*,0}(X \times EC_2) \arrow[r, hook] & H^{*,*}(X \times EC_2) \\
(\tau^{-1}H)^{*,0}(X) \arrow[u, dashed, "\cong"] \arrow[r, hook] & \tau^{-1}H^{*,*}(X) \arrow[swap, u, "\varphi"]
\end{tikzcd}
\end{center}
Finally, by Lemma \ref{quotient lemma} the cohomology on the $p$-axis is the cohomology of the quotient, so we have the isomorphism
\[
H^{*,0}(X \times EC_2) \cong H^*_{sing}(X \times_{C_2} EC_2) = H^*_{Bor}(X)
\]
and this completes the proof.
\end{proof}

Applying $\tau$-localization to our main theorem we have
\begin{align*}
\tau^{-1}H^{*,*}(X) &\cong \tau^{-1}\left((\oplus_i \Sigma^{p_i,q_i} \Mt) \oplus (\oplus_j \Sigma^{r_j,0} \A_{n_j}) \right)\\
&\cong \left(\oplus_i \Sigma^{p_i,0} \A_{\infty} \right) \oplus \left( \oplus_j  \Sigma^{r_j,0} \A_{n_j} \right).
\end{align*}
So identifying $x$ with $\tau^{-1} \rho$, on the $p$-axis we have
\[
(\tau^{-1}H)^{*,0}(X) \cong \left(\oplus_i \Sigma^{p_i} \F_2[x] \right) \oplus \left(\oplus_j \Sigma^{r_j} \F_2[x]/(x^{n_j+1})\right).
\]
Now compare this with Borel cohomology using Lemma \ref{Borel} and the decomposition of $H^*_{Bor}(X)$ as a $\F_2[x]$-module.  We see the torsion components of Borel cohomology correspond precisely to the shifted copies of $\A_{n_j}$ in $H^{*,*}(X)$, and the free components correspond to shifted copies of $\Mt$.  In particular, if we know the Borel cohomology of a finite $C_2$-CW complex, to compute $RO(C_2)$-graded cohomology we only need to determine the weight of each copy of $\Mt$. This is often nontrivial, though.


\appendix


\section{Injectivity}\label{Injectivity}

The purpose of this section is to prove that $\Mt$ is self-injective.  The proof will use a graded version of Baer's criterion (Proposition 9.3.6 in \cite{B}).  According to Baer's criterion, $\Mt$ is injective if and only if for every graded ideal $J \subseteq \Mt$, any map $f:\Sigma^{p,q} J \to \Mt$ extends to a map $\bar{f}: \Sigma^{p,q}\Mt \to \Mt$ as in the diagram below.
\begin{center}
\begin{tikzcd}
\Sigma^{p,q} J \arrow[r,"f"] \arrow[hook, d] & \Mt \\
\Sigma^{p,q} \Mt \arrow[dashed, ur, swap,"\bar{f}"] &
\end{tikzcd}
\end{center}
Equivalently, it suffices to show that for every map $f:\Sigma^{p,q} J \to \Mt$ there is an element $\lambda \in \Mt$ such that $f(x) = \lambda x$ for all $x \in J$.

In order to show that Baer's criterion is satisfied, we first need to investigate ideals of $\Mt$.  A few examples are shown in Figure \ref{ideals}.  The first ideal pictured is generated by three elements in the upper cone and contains the entire lower cone.  The second ideal is generated by 3 elements in the lower cone.  The last two ideals shown here are infinitely generated by elements in the lower cone.

\begin{figure}[ht]
\begin{center}
\begin{tikzpicture}[scale=0.25]
\draw[gray] (-5.5,0) -- (5.5,0);
\draw[gray] (0,-6.5) -- (0,5.5);

\draw[thick, gray] (0,0) -- (5.5,5.5);
\draw[thick, gray] (0,0) -- (0,5.5);
\draw[thick, gray] (0,-2) -- (0,-6.5);
\draw[thick, gray] (0,-2) -- (-4.5,-6.5);

\fill[lightgray] (0,5.5) -- (0,3) -- (2,5) -- (2,4) -- (3,5) -- (3,4) -- (4.5,5.5) -- cycle;
\fill[lightgray] (-4.5,-6.5) -- (0,-2) -- (0,-6.5) -- cycle;
\draw[thick] (0,5.5) -- (0,3) -- (2,5) -- (2,4) -- (3,5) -- (3,4) -- (4.5,5.5);
\draw[thick] (-4.5,-6.5) -- (0,-2) -- (0,-6.5);

\end{tikzpicture} \hfill
\begin{tikzpicture}[scale=0.25]
\draw[gray] (-5.5,0) -- (5.5,0);
\draw[gray] (0,-6.5) -- (0,5.5);

\draw[thick, gray] (0,0) -- (5.5,5.5);
\draw[thick, gray] (0,0) -- (0,5.5);
\draw[thick, gray] (0,-2) -- (0,-6.5);
\draw[thick, gray] (0,-2) -- (-4.5,-6.5);

\fill[lightgray] (-1,-6) -- (-1,-5) -- (-2,-6) -- (-2,-5) -- (-3,-6) -- (-3,-5) -- (0,-2) -- (0,-5) -- cycle;
\draw[thick] (-1,-6) -- (-1,-5) -- (-2,-6) -- (-2,-5) -- (-3,-6) -- (-3,-5) -- (0,-2) -- (0,-5) -- cycle;

\end{tikzpicture}\hfill
\begin{tikzpicture}[scale=0.25]
\draw[gray] (-5.5,0) -- (5.5,0);
\draw[gray] (0,-6.5) -- (0,5.5);

\draw[thick, gray] (0,0) -- (5.5,5.5);
\draw[thick, gray] (0,0) -- (0,5.5);
\draw[thick, gray] (0,-2) -- (0,-6.5);
\draw[thick, gray] (0,-2) -- (-4.5,-6.5);

\fill[lightgray] (-4.5,-6.5) -- (0,-2) -- (0,-6.5) -- cycle;
\draw[thick] (-4.5,-6.5) -- (0,-2) -- (0,-6.5);

\end{tikzpicture} \hfill
\begin{tikzpicture}[scale=0.25]
\draw[gray] (-5.5,0) -- (5.5,0);
\draw[gray] (0,-6.5) -- (0,5.5);

\draw[thick, gray] (0,0) -- (5.5,5.5);
\draw[thick, gray] (0,0) -- (0,5.5);
\draw[thick, gray] (0,-2) -- (0,-6.5);
\draw[thick, gray] (0,-2) -- (-4.5,-6.5);

\fill[lightgray] (-4.5,-6.5) -- (0,-2) -- (0,-5) -- (-1,-6) -- (-1,-5) -- (-2,-6) -- (-2,-5) -- (-3.5,-6.5) -- cycle;
\draw[thick] (-4.5,-6.5) -- (0,-2) -- (0,-5) -- (-1,-6) -- (-1,-5) -- (-2,-6) -- (-2,-5) -- (-3.5,-6.5);
\end{tikzpicture}
\end{center}
\caption{Some ideals in $\Mt$.}
\label{ideals}
\end{figure}

For our purposes it will be useful to classify graded ideals of $\Mt$ as one of two types.  We use the notation $J^+ = \Mt^+ \cap J$ and $J^- = \Mt^- \cap J$.

\begin{lemma}\label{ideal classification}
Every graded ideal $J \subseteq \Mt$ is one of the following two types:
\begin{enumerate}[I.]
\item $J$ is finitely generated by homogeneous elements $x_1, \dots, x_n$ with each $x_i \in \Mt^+$ and $J^- = \Mt^-$; or
\item $J^+ = 0$.
\end{enumerate}
\end{lemma}

\begin{proof}
Observe that if $J$ contains a nonzero homogeneous element of $\Mt^+$, i.e.\ there exists $x = \rho^m \tau^n \in J$ for some $m, n \geq 0$, then $\Mt^- \subseteq J$.  
For any $a,b \geq 0$, the element $\frac{\theta}{\rho^a \tau^b} \in J$ because
\[
\frac{\theta}{\rho^a \tau^b} = \frac{\theta}{\rho^{m+a} \tau^{n+b}} \cdot \rho^m \tau^n = \frac{\theta}{\rho^{m+a} \tau^{n+b}} \cdot x \in J.
\]
So indeed $\Mt^- \subseteq J$.  Since $\Mt^+ \cong \F_2[\rho, \tau]$ is a graded polynomial ring, any graded ideal of this form is finitely generated by some number of homogeneous elements in $\Mt^+$.  Thus $J$ is an ideal of type I.

Alternatively, if $J^+ = 0$, then $J$ only contains elements of $\Mt^-$.
\end{proof}

Returning to Figure \ref{ideals}, we see the first ideal pictured is type I and the rest are type II.  Notice an ideal of type II may be finitely generated or infinitely generated.

Next we will describe all possible nonzero maps $f:\Sigma^{p,q} J \to \Mt$.  We will see there are not really any interesting maps.  Any such $f$ is completely determined by $(p,q)$ the bidegree of the suspension.

\begin{lemma}\label{map classification}
Let $J \subseteq \Mt$ be a graded ideal and $f:\Sigma^{p,q} J \to \Mt$ be a nontrivial map.  Then exactly one of the following holds,
\begin{enumerate}[I.]
\item $J$ is type I and either
\begin{enumerate}[(i)]
\item $f(\rho^m \tau^n) = \rho^a \tau^b$ for some $m, n \geq 0$ and $a \geq m$, $b \geq n$; or
\item $f(\rho^m \tau^n) = \frac{\theta}{\rho^a \tau^b}$ for some $m, n \geq 0$ and $a,b \geq 0$.
\end{enumerate}
\item $J$ is type II and $f\big(\frac{\theta}{\rho^m \tau^n}\big) = \frac{\theta}{\rho^a \tau^b}$ for some $m, n \geq 0$ and $m \geq a \geq 0$,
$n \geq b \geq 0$.
\end{enumerate}
\end{lemma}

\begin{proof}
We will consider each type of ideal separately.
\begin{enumerate}[I.]
\item We begin by assuming $J$ is type I so that $J^+$ is nontrivial and $J^- = \Mt^-$.  If $f(J^+)= 0$ then $f(\rho^m \tau^n)=0$ for all $\rho^m \tau^n \in J^+$.  But then for any $a,b \geq 0$
\[
f\left( \frac{\theta}{\rho^a \tau^b} \right) = f\left( \frac{\theta}{\rho^{a+m} \tau^{b+n}} \cdot \rho^m \tau^n \right) = \frac{\theta}{\rho^{a+m} \tau^{b+n}} \cdot f(\rho^m \tau^n) = 0.
\]
So we also have $f(J^-) = 0$, contradicting that $f$ is nontrivial.  Hence there must be some $m,n \geq0$ such that $\rho^m \tau^n \in J^+$ and $f(\rho^m \tau^n) \neq 0$.  Now there are two cases, either $f(\rho^m \tau^n) \in \Mt^+$ or $f(\rho^m \tau^n) \in \Mt^-$.
\begin{enumerate}[(i)]
\item Suppose $f(\rho^m \tau^n) \in \Mt^+$ so that $f(\rho^m \tau^n) = \rho^a \tau^b$ for some $a, b \geq 0$.  We just need to show that $a \geq m$ and $b \geq n$.  This follows immediately because if either $a < m$ or $b < n$ then
\[
0 = f(0) = f\left( \frac{\theta}{\rho^a \tau^b} \cdot \rho^m \tau^n \right) = \frac{\theta}{\rho^a \tau^b} \cdot f(\rho^m \tau^n) = \theta
\]
in $\Mt$, a contradiction.
\item Now suppose $f(\rho^m \tau^n) \in \Mt^-$.  Then $f(\rho^m \tau^n) = \frac{\theta}{\rho^a \tau^b}$ for some $a,b \geq 0$.  There is no further restriction on the values of $a$ and $b$ in this case.
\end{enumerate}
\item Next we consider an ideal of type II, so that $J^+ = 0$. Since $f$ is nontrivial, there must be some $m, n \geq 0$ with  $f\big(\frac{\theta}{\rho^m \tau^n}\big) \neq 0$.
It is not possible that $f\big(\frac{\theta}{\rho^m \tau^n}\big) \in \Mt^+$ because if $f\big(\frac{\theta}{\rho^m \tau^n}\big) = \rho^a \tau^b$ for some $a, b \geq 0$, then we get an immediate contradiction
\[
0 = f\left( \rho^{m+1}\tau^{n+1} \cdot \frac{\theta}{\rho^m \tau^n} \right) = \rho^{m+1}\tau^{n+1} \cdot f\left(\frac{\theta}{\rho^m \tau^n}\right) = \rho^{a+m+1}\tau^{b+n+1}.
\]
So it must be the case that $f\big(\frac{\theta}{\rho^m \tau^n}\big) \in \Mt^-$.  Then $f\big(\frac{\theta}{\rho^m \tau^n}\big) = \frac{\theta}{\rho^a \tau^b}$ for some $m, n \geq 0$ and $a,b \geq 0$.  It remains to show that $m \geq a$ and that $n \geq b$.
If either $m < a$ or $n < b$, then again we get a contradiction because
\[
0 = f(0) = f\left( \rho^{a}\tau^{b} \cdot \frac{\theta}{\rho^m \tau^n} \right) = \rho^{a}\tau^{b} \cdot f\left(\frac{\theta}{\rho^m \tau^n} \right) = \rho^{a}\tau^{b} \cdot \frac{\theta}{\rho^a \tau^b} = \theta.
\]
\end{enumerate}
This completes the proof.
\end{proof}

It appears we have not fully described each map in the previous lemma since we only described the image of a single element.  The following lemma implies that any map of a graded ideal $f:\Sigma^{p,q} J \to \Mt$ is completely determined by a single nonzero element in its image.  Hence, Lemma \ref{map classification} does indeed classify maps of graded ideals to $\Mt$.

\begin{lemma}\label{unique}
Let $J \subseteq \Mt$ be a graded ideal with two maps $f,g: \Sigma^{p,q} J \to \Mt$.  If $x \in J$ is a nonzero homogeneous element with $f(x) \neq 0$ and $f(x) = g(x)$, then $f = g$.
\end{lemma}

\begin{proof}
We begin by observing that $\Mt$ has at most one nonzero element in any given bidegree.  This implies that if $m$, $x$, and $y$ are homogeneous elements of $\Mt$ with $mx \neq 0$ and $mx = my$, then $x = y$.  We call this property $P$.

Let $A = \{y \in J \mid f(y) = g(y)\}$.  Our goal is to show that $A = J$.  Notice that $A \subseteq J$ is a submodule because $f$ and $g$ are both $\Mt$-module maps.  Furthermore, if $m \in \Mt$ and $y \in J$ are homogeneous elements with $my \in A$ and $f(my) \neq 0$, then $y \in A$.  This is because
\[
mf(y)=f(my)=g(my)=mg(y)
\]
and by property $P$ we must have $f(y)=g(y)$.  We now proceed with several cases, considering each type of ideal separately.
\begin{enumerate}[I.]
\item Suppose $J$ is type I.  Then $J$ is finitely generated by some elements of $\Mt^+$ and $J^-=\Mt^-$.
\begin{enumerate}
\item Suppose $x \in J^+$.  The ideal generated by $x$ is type I so it contains $\Mt^-$.  Since $A$ is a submodule and $x$ is in $A$ by assumption, $J^- = \Mt^- \subseteq A$.  Consider a homogeneous element $y \in J-A$, which would have to be in $J^+$.  Set $z=\lcm(x,y)$ so that $z = \rho^a \tau^b y$ for some $a,b \geq 0$.  Either $f(x) \in \Mt^+$ or $f(x) \in \Mt^-$.
Using $z$, we can show in either case that in fact $y$ must be in $A$.
\begin{enumerate}
\item Suppose $f(x) \in \Mt^+$.  Since $z \in A$, $f(z) = g(z)$.  There is no $\rho$- or $\tau$-torsion in $\Mt^+$ so $f(z) \neq 0$.  By property $P$ we see that $f(y) = g(y)$ and indeed $y \in A$.
\item Suppose $f(x) \in \Mt^-$.  Again $f(z) = g(z)$ since $z \in A$.  It is possible that $f(z) = 0$ for degree reasons, in which case $f(y)=g(y) = 0$ for degree reasons as well.  Otherwise, $f(z) \neq 0$ and again by property $P$ we have $f(y)=g(y)$.
\end{enumerate}
\item Now suppose that $x \in J^-$ (though we still assume $J$ is type I).  Then we can write $x = mx'$ for some $x' \in J^+$ and $m \in \Mt^-$.  As usual, since $f(x) = f(mx') \neq 0$ and $mx' \in A$, property $P$ implies that $f(x') = g(x')$.  Now $x'$ satisfies case (a) above.  From the previous argument we have that $A = J$.
\end{enumerate}
\item Suppose $J$ is type II.  Then $x \in J^-$ because $J^+=0$.  Suppose $y \in J-A$ is a nonzero homogeneous element.  Set $z = \lcm(x,y)$ so that $z = \rho^a \tau^b y$ for some $a,b \geq 0$ and repeat the argument in case (a)(ii) above.  From the proof of Lemma \ref{map classification} it is not possible that $f(x) \in \Mt^+$, so we are done.
\end{enumerate}
This completes the proof that $A = J$ and so $f = g$.
\end{proof}

We are now ready to prove that $\Mt$ is self-injective.  The following proposition also appears as Proposition \ref{injectivity prop}.

\begin{prop}
The regular module $\Mt$ is injective as an $\Mt$-module.
\end{prop}

\begin{proof}
From the discussion of the graded version of Baer's criterion above, it suffices to show that for any graded ideal $J \subseteq \Mt$ and any map $f:\Sigma^{p,q} J \to \Mt$, there is an element $\lambda \in \Mt$ such that $f(x) = \lambda x$ for all $x \in J$.  Let $J$ be an ideal and $f:\Sigma^{p,q} J \to \Mt$.  Of course, if $f = 0$ then we can take $\lambda = 0$, so assume $f \neq 0$.
Every nontrivial map $f$ satisfies one of the cases described in Lemma \ref{map classification}, so we consider each separately.
\begin{enumerate}[I.]
\item Assume the ideal $J \subseteq \Mt$ is type I so that $J^+$ contains some elements of $\Mt^+$ and $J^- = \Mt^-$.
\begin{enumerate}[(i)]
\item Suppose $f(\rho^m \tau^n) = \rho^a \tau^b$ for some $m, n \geq 0$ and $a \geq m$, $b \geq n$.  Define $g:\Sigma^{p,q} J \to \Mt$ to be multiplication by $\lambda = \rho^{a-m} \tau^{b-n}$ so we have $g(x) = \rho^{a-m} \tau^{b-n} \cdot x$ for all $x \in J$.
Then $f(\rho^m \tau^n) = g(\rho^m \tau^n)$ and by uniqueness from Lemma \ref{unique} we have that $f = g$.
\item Suppose $f(\rho^m \tau^n) = \frac{\theta}{\rho^a \tau^b}$ for some $m, n \geq 0$ and $a,b \geq 0$.  Now we define $g:\Sigma^{p,q} J \to \Mt$ to be multiplication by $\lambda = \frac{\theta}{\rho^{a+m} \tau^{b+n}}$.
Again $f(\rho^m \tau^n) = g(\rho^m \tau^n)$ and by uniqueness $f = g$.
\end{enumerate}
\item Finally assume the ideal $J \subseteq \Mt$ is type II so that $J^+ = 0$.  We know in this case $f\big(\frac{\theta}{\rho^m \tau^n}\big) = \frac{\theta}{\rho^a \tau^b}$ for some $m, n \geq 0$ and $m \geq a \geq 0$,
$n \geq b \geq 0$.  Now $f$ agrees with multiplication by $\lambda = \rho^{m-a} \tau^{n-b}$.
\end{enumerate}
We have shown every map $f:\Sigma^{p,q} J \to \Mt$ is multiplication by some element $\lambda \in \Mt$, which completes the proof that the regular module is injective.
\end{proof}

\newpage


\bibliographystyle{abbrv}
\bibliography{refs}
\nocite{*}

\end{document}